\documentclass[a4paper, 11pt]{article}

\usepackage{authblk}

\usepackage{amsmath,amsthm,amssymb,enumerate}
\usepackage[margin=2cm]{geometry}
\usepackage[hidelinks]{hyperref}
\usepackage[capitalise]{cleveref}
\usepackage{tikz}
\usetikzlibrary{calc,intersections,through,backgrounds,arrows,patterns}
\usepackage{subcaption}

\tikzstyle{grid_point}=[circle,draw=black!50,text=white,inner sep=0.5mm,font=\scriptsize]
\tikzstyle{grid_white_point}=[circle,draw=black!50,fill=white,text=white,inner sep=0.5mm,font=\scriptsize]
\tikzstyle{true_point}=[circle,draw=black!60,fill=black!60,text=white,inner sep=0.5mm,font=\scriptsize]
\tikzstyle{false_point}=[circle,draw=black!70,text=white,inner sep=0.5mm,font=\scriptsize]
\tikzstyle{intersection_point}=[circle,draw=black!90,fill=black!90,text=white,inner sep=0.4mm,font=\scriptsize]

\tikzstyle{oriented_segment} = [draw,line width=1pt,arrows=-latex',black!80]
\tikzstyle{dashed_line} = [draw,dashed,line width=0.4pt,black!60]

\makeatletter
\tikzset{
        hatch distance/.store in=\hatchdistance,
        hatch distance=5pt,
        hatch thickness/.store in=\hatchthickness,
        hatch thickness=5pt
        }
\pgfdeclarepatternformonly[\hatchdistance,\hatchthickness]{north east hatch}
    {\pgfqpoint{-1pt}{-1pt}}
    {\pgfqpoint{\hatchdistance}{\hatchdistance}}
    {\pgfpoint{\hatchdistance-1pt}{\hatchdistance-1pt}}%
    {
        \pgfsetcolor{\tikz@pattern@color}
        \pgfsetlinewidth{\hatchthickness}
        \pgfpathmoveto{\pgfqpoint{0pt}{0pt}}
        \pgfpathlineto{\pgfqpoint{\hatchdistance}{\hatchdistance}}
        \pgfusepath{stroke}
    }
\pgfdeclarepatternformonly[\hatchdistance,\hatchthickness]{north west hatch}
    {\pgfqpoint{-1pt}{-1pt}}
    {\pgfqpoint{\hatchdistance}{\hatchdistance}}
    {\pgfpoint{\hatchdistance-1pt}{\hatchdistance-1pt}}%
    {
        \pgfsetcolor{\tikz@pattern@color}
        \pgfsetlinewidth{\hatchthickness}
        \pgfpathmoveto{\pgfqpoint{0pt}{\hatchdistance}}
        \pgfpathlineto{\pgfqpoint{\hatchdistance}{0pt}}
        \pgfusepath{stroke}
    }
\makeatother

\tikzstyle{small_text} = [font=\scriptsize]

\theoremstyle{plain}
\newtheorem{theorem}{Theorem}

\newtheorem{corollary}[theorem]{Corollary}

\newtheorem{lemma}[theorem]{Lemma}

\newtheorem{claim}[theorem]{Claim}
\theoremstyle{definition}
\newtheorem{definition}{Definition}[section]

\newcommand{\Conv}{\textup{Conv}}
\newcommand{\Vertic}{\textup{Vert}}
\newcommand{\dir}{\overrightarrow}
\newcommand{\Area}{\textup{Area}}
\newcommand{\superb}{proper }

\newcommand{\rMN}{\mathcal{R}_{u,v}}

\newcommand{\gridMN}{\mathcal{G}_{m,n}}

\newcommand{\grid}{\mathcal{G}}
\newcommand{\m}{u}
\newcommand{\n}{v}
\newcommand{\tr}[3]{{#1}{#2}{#3}}
\newcommand{\ortr}[3]{\overrightarrow{{#1}{#2}{#3}}}

\title{Asymptotics of the number of 2-threshold functions}
\author[1]{Elena Zamaraeva}
\author[2]{Jovi{\v{s}}a {\v{Z}}uni{\'c}}
\affil[1]{Mathematical Center, Lobachevsky State University of Nizhni Novgorod, Russia}
\affil[2]{Mathematical Institute, Serbian Academy of Sciences, Serbia}

\definecolor{light_grey}{rgb}{0.85,0.85,0.85}
\definecolor{light_light_grey}{rgb}{0.95,0.95,0.95}

\begin{document}

\maketitle

\begin{abstract}

A $k$-threshold function on a rectangular grid of size $m \times n$ is the conjunction of $k$ threshold functions on the same domain.
In this paper, we focus on the case $k=2$ and show that the number of two-dimensional 2-threshold functions is~$\dfrac{25}{12\pi^4} m^4 n^4 + o(m^4n^4)$.
\end{abstract}


\tableofcontents
\section{Introduction}
Let $\gridMN$ denote an integer two-dimensional rectangular grid, that is, $\gridMN = \{0,\dots,m-1\} \times \{0,\dots,n-1\}$.
For a $\{0,1\}$-valued function $f$ defined on $\gridMN$ we denote 
$$
M_\nu(f) = \{ x \in \gridMN | f(x) = \nu\},
$$
where $\nu \in \{0,1\}$.
For a given set of points $S$ we denote by $\Conv(S)$ the convex hull of $S$.

We say that a $\{0,1\}$-valued function $f$ defined on the grid $\gridMN$ is \textit{threshold} if its sets of true and false points are separable by a line, i.e. 
$$
\Conv(M_0(f)) \cap \Conv(M_1(f)) = \emptyset.
$$
Let $k$ be a natural number, a function $f :\gridMN \rightarrow \{0,1\}$ is called \textit{$k$-threshold} if there exist at most $k$ threshold functions $f_1, \dots, f_k$ such that
$$
M_1(f) = \bigcap_{i=1}^{k} M_1(f_i).
$$
We say that $f$ is a \emph{conjunction} of $f_1, \dots, f_k$, i.e. $f = f_1 \land \dots \land f_k$.
We also say that the functions $f_1,\dots,f_k$ \emph{define} the $k$-threshold function $f$.
A $k$-threshold function is called \emph{proper $k$-threshold} if it is not $(k-1)$-threshold.

In this work we focus on $2$-threshold functions, i.e. the conjunctions of two threshold functions.
Some previous works on this topic dealt with learning issues (\cite{Hegedus1997,Zamaraeva2016, Zamaraeva2017}) and the structure (\cite{Part1}) of $2$-threshold functions whereas this paper is devoted to the number of $2$-threshold functions.
Denote by $t_k(m,n)$ the number of $k$-threshold functions on $\gridMN$.
Throughout the paper we will write $t(m,n)$ instead of $t_1(m,n)$, as the former is a common notation in the literature.
The main result of the paper is the following

\begin{theorem}
	\label{th:t_k}
	$t_2(m,n) = \frac{25}{12\pi^4} m^4 n^4 + o(m^4n^4)$.
\end{theorem}

The study goes back to Gauss \cite{gauss} and study of the number of lattice points inside a circle, but someone may refer to earlier times. 
Configurations of planar lattice points inside a circle are also considered in \cite{huxley-focm,huxley-lms} and configurations of lattice points inside a sphere are considered in \cite{zunic-aam}. 
Configurations produced by straight lines are also intensively studied  \cite{1,Koplowitz,Haukkanen,Zolotykh,citation-2,6}.
This is partially because of their applications in digital geometry and computer graphics. 
Problems related to the planar lattice points configurations produced by convex curves and arbitrary curves are considered in \cite{ivic}
and \cite{9}, respectively. Configurations by multiple curves (surfaces) are of an interest too. These are 
observed \cite{Zamaraeva2016,Zamaraeva2017,zunic-ieee}, mainly because of
applications in the neural networks and machine learning theory \cite{gpz,zunic-ieee}.
Further attempts to get richer structures that can be used for data analysis (e.g. the related object classification) based on machine 
learning approaches, can be obtained by increasing the space dimension. A use of polynomials \cite{anthony}, 
or particular subsets of them \cite{bruck},  
of a large number of variables  has been considered already. Treating and computing Boolean functions by suitable selection of 
polynomial threshold functions has also been of interest \cite{boolean,boolean-2}.

In the majority of approaches where the multiple lines and surfaces were used, the partitioning surfaces are not assumed to be in general position 
(e.g. parallel hyperplanes or $d$-dimensional spheres centered at the same point). The situation, once the data space partition surfaces are in general position,
can be much more difficult, especially for a proper  performance analysis of the  methods proposed. Herein, as 
it has been mentioned above, we consider one of such problems.

The asymptotics of the number of threshold functions for square grids was first obtained in \cite{Koplowitz}:
$$
t(n,n) = \frac{6}{\pi^2}n^4 + O(n^3  \log n),
$$
and for arbitrary rectangular grids in \cite{Acketa}:
$$
t(m,n) = \frac{6}{\pi^2}m^2n^2 + O(m^2 n \log n \ +  \  mn^2 \log \log n),
$$
where $m < n$ is assumed.

The current best known formula was obtained in \cite{Haukkanen}:
$$
t(m,n) = \frac{6}{\pi^2}m^2n^2 + O(mn^2).
$$

An important point to note here is that all the above results are based on the relation between non-constant threshold functions and segments with specific properties, and similar methods are used in this paper.

The above asymptotics provides a trivial upper bound on the number of $k$-threshold functions for a fixed $k > 1$:

\begin{flalign}
\label{eq:k_trivial_upper_bound}
\nonumber t_k(m,n) \leq {t(m,n) \choose k} &= \frac{t(m,n)^k}{k!} + O\left(t(m,n)^{k-1}\right)\\
&= \frac{6^k}{\pi^{2k}k!}m^{2k}n^{2k} + O\left(m^{2k-1}n^{2k}\right).
\end{flalign}
However, no asymptotics was known for the number of $k$-threshold functions for any $k > 1$. 

In order to prove \cref{th:t_k}, we will first make use of the structural characterization of $2$-threshold functions from \cite{Part1} to reduce the problem to the enumeration of pairs of prime segments in convex position. 
Then we will derive the asymptotic formula for the number of such pairs.


The organization of the paper is as follows.
All preliminary information can be found in \cref{sec:preliminaries}.
In \cref{sec:pair_oriented_segments} we recall the results from \cite{Part1} introducing \emph{proper} pairs of segments and describing their relation to $2$-threshold functions.
We also express the number of $2$-threshold functions through the number of proper pairs of segments.
In \cref{sec:from2thresholdToProperPairs} we reveal the relation between proper pairs of segments and pairs of prime segments in convex position and reduce the problem of the estimation of the number of $2$-threshold functions to that of pairs of prime segments in convex position.
Finally, \cref{sec:numberOfProperPairs} is devoted to the estimation of the number of pairs of prime segments in convex position.
In \cref{sec:k-threshold-number} we use the obtained formula to improve the trivial upper bound on the number of $k$-threshold functions (\ref{eq:k_trivial_upper_bound}) for $k \geq 3$.

\section{Preliminaries}
\label{sec:preliminaries}

In this paper we use capital letters $A,B,C$ etc. to denote points on the plane. 
The distance between two points $A$ and $B$ is denoted by $d(A,B)$.
The distance between two sets of points $S_1$, $S_2$ is the minimum distance between two points $A \in S_1$ and $B \in S_2$ and denoted by $d(S_1,S_2)$.
The distance between a point and a set of points is denoted analogously.
The line passing through two distinct points $A$ and $B$ is denoted by $\ell(AB)$.
For a convex polygon $\mathcal P$ we denote by $\Area(\mathcal P)$ the area of $\mathcal P$.

We say that a point $A = (x,y)$ is \emph{integer}, if both of its coordinates $x$ and $y$ are integer.
If two distinct integer points $A$, $B$ are the only integer points on the segment $AB$, we say that the segment $AB$ is \textit{prime}.

For two integers $p$ and $q$ we will write $p \perp q$ if they are coprime, i.e. the greatest common divisor of $p$ and $q$ is $1$.

For a polygon $P$ denote by $\Vertic(P)$ the set of vertices of $P$. 
We say that the points $A_1, A_2, \dots, A_n$ are in convex position if $\{A_1,\dots,A_n\} = \Vertic(\Conv(\{A_1,\dots,A_n\}))$.
We will say that two segments are \emph{in convex position} if they are opposite sides of a convex quadrilateral.
We also denote by $P(f)$ the convex hull of true points of $f$, that is $P(f) = \textup{Conv}(M_1(f))$.

Let $\mathcal C$ be a convex set. 
A convex polygon $\mathcal P$ is called \textit{circumscribed} about $\mathcal C$ if for every edge $AB$ of $\mathcal P$ the line $\ell(AB)$ is a tangent to $\mathcal C$ and $AB \cap \mathcal{C} \neq \emptyset$.
We denote by $B(\gridMN)$ the set of points of $\gridMN$ belonging to its boundary, i.e.
$$
B(\gridMN) = \{(x_1,x_2) \in \gridMN | x_1 \in \{0, m-1\} \text{ or }  x_2 \in \{0,n-1\}\}.
$$

\subsection{Segments, triangles, quadrilaterals and their orientation}
\label{sec:seg_tr_orientation}

We often denote a \emph{convex} polygon 
by a sequence of its vertices in either clockwise or counterclockwise order. 
For instance, we denote by
\begin{itemize}
\item [$AB$] the segment with endpoints $A,B$;
\item [$ABC$] the triangle with vertices $A,B,C$;
\item [$ABCD$] the convex quadrilateral with edges $AB$, $BC$, $CD$, $DA$.
\end{itemize}

We call a polygon or a segment \emph{oriented} and add an arrow in the notation if the order of vertices is important. 
For example, $\dir{AB}$, $\dir{ABC}$, $\dir{ABCD}$ denote the oriented segment, the oriented triangle, and the oriented convex quadrilateral, respectively.
An oriented convex polygon is \textit{clockwise} or \textit{counterclockwise} depending on the orientation of the rotation.
If $\dir{ABC}$ is a clockwise triangle, then $\dir{ACB}$ is a counterclockwise triangle and vice versa.

\subsection{Number theoretic preliminaries}
\label{sec:numberTheoryPrlim}

In the subsequent sections we will use the following formulas.
For the harmonic number:
\begin{equation}
\label{eq:1/i}
\sum_{i=1}^{n}\frac{1}{i}=\log n + \gamma + O\left( \frac{1}{n} \right) = \log n + O(1),
\end{equation}
where $\gamma$ is the Euler-Mascheroni constant.

For a fixed natural $k$ the asymptotics of the sum of $k$-th powers can be estimated as
\begin{equation}
\label{eq:sum_powers}
\sum_{i=1}^{n} i^k = \frac{n^{k+1}}{k+1} + O(n^{k}).
\end{equation}

Also, for a fixed natural $q$ and integer $k\geq 0$ we have
\begin{equation}
\label{eq:sum_euler_gen}
\sum_{\substack{p = 1 \\ p \perp q}}^{n} p^k = \frac{\phi (q)}{q}\frac{n^{k+1}}{(k+1)} + O(n^k2^{w(q)}),
\end{equation}
where $\phi(q)$ is the Euler function, $w(q)$ is the number of different prime divisors of $q$ and

\begin{equation}
\label{eq:sum_w}
\sum_{q =1}^{n} O(2^{w(q)})= O(n \log n).
\end{equation}

For the negative powers of $p$ we have

\begin{equation}
\label{eq:sum_euler_divide_new}
\sum_{\substack{p = 1 \\ p \perp q}}^{n} \frac{1}{p} = \frac{\phi (q)}{q}\log n  + \frac{\phi(q)}{q}\gamma + \frac{\phi(q)}{q}O\left(\frac{1}{n}\right) - \sum_{d|q}\mu(d)\frac{1}{d}\log d.
\end{equation}

Some sums regarding the Euler function are as follows:
\begin{equation}
\label{eq:sum_phi_ln}
\sum_{x=1}^{n}\phi(x)\log(x) = \frac{3}{\pi^2}n^2\log n - \frac{3}{2\pi^2}n^2 + o(n^2).
\end{equation}

The general formula for the power of $x$ follows:

\begin{equation}
\label{eq:sum_phi_n_k}
\sum_{x=1}^{n}\phi(x)x^k = \frac{6}{\pi^2}\frac{n^{k+2}}{(k+2)} + O(n^{k+1} \log n),
\end{equation}
where $k$ is integer.

More details about the derivations of the previous sums are provided in \cite{Euler} and \cite{apostol}.

The following sum is obtained from (\ref{eq:sum_euler_gen}) and (\ref{eq:sum_phi_n_k}):

\begin{flalign}
\label{eq:sum_coprime}
\sum_{p = 1}^{m}\sum_{\substack{q = 1 \\ q \perp p}}^{n}1 = \frac{6}{\pi^2} mn + O(m \log n).
\end{flalign}

For the M\"{o}bius function $\mu_{n}$ we have (\cite{apostol}):
\begin{equation}
\label{eq:mobius}
    \sum\limits_{n | k} \mu_{n} \ = \ 
		\left\{ 
		\begin{array}{cl}
		1 & \makebox{if} \ \  k = 1 \\[0.2cm]  
		0 & \makebox{if} \ \  k > 1. 
		\end{array}
		\right.
\end{equation}
where $n | k$ means that $n$ is a divisor of $k$, and
\begin{equation}
\label{eq:mobius_riemann} 
\sum_{n=1}^{\infty} \frac{\mu(n)}{n^2} \ = 
\frac{1}{\zeta(2)} \ =  \frac{6}{\pi^2},
\end{equation}
where $ \zeta(n)$ is the Riemann zeta function.

\section{Proper pairs of oriented prime segments and the number of $2$-threshold functions}
\label{sec:pair_oriented_segments}

The known asymptotic formulas for the number of threshold functions (\cite{Koplowitz, Acketa, Haukkanen}) are based on their relation to (oriented) prime segments.

\begin{definition}\label{def:segment_defines_f}
Let $AB$ be a prime segment and $A,B \in \gridMN$.
We say that $\dir{AB}$ \textit{defines} a $\{0,1\}$-valued function $f$ on $\gridMN$ if:
\begin{enumerate}
\item $f(A) = 1, f(B) = 0$;
\item for any $X \in \gridMN \cap \ell(AB)$ we have $f(X) = 1$ if and only if $d(A,X) < d(B,X)$;
\item for any $X \in \gridMN \setminus \ell(AB)$ we have $f(X) = 1$ if and only if $\ortr{A}{B}{X}$ is a counterclockwise triangle.
\end{enumerate}
The function defined by $\dir{AB}$ is denoted as $f_{\dir{AB}}$.
\end{definition}

It was proved in \cite{Koplowitz} that the function $f_{\dir{AB}}$ defined by an oriented prime segment $\dir{AB}$ is threshold.
Moreover, in the same paper, a bijection between all non-constant threshold functions and oriented prime segments in $\gridMN$ was established.

Following a similar approach with $2$-threshold functions, we introduced the next definition in~\cite{Part1}:

\begin{definition}
We say that a pair of oriented prime segments $\dir{AB}, \dir{CD}$ in $\gridMN$ \textit{defines} a $2$-threshold function $f$ on $\gridMN$ if
$$
f = f_{\dir{AB}} \land f_{\dir{CD}}.
$$
\end{definition}

However, since the same $2$-threshold function can be represented as the conjunction of different pairs of threshold functions, it can be defined by different pairs of oriented prime segments.
To overcome this difficulty, we imposed an extra restriction on pairs of segments, which resulted in the notion of \emph{proper} pairs of segments \cite{Part1}.

\begin{definition}
We say that a pair of oriented segments $\dir{AB}, \dir{CD}$ is \emph{proper} if the segments are prime and 
$$
f_{\dir{CD}}(A) = f_{\dir{CD}}(B) = f_{\dir{AB}}(C) = f_{\dir{AB}}(D) = 1.
$$
\end{definition}

In  \cite{Part1} we established the following connections between proper $2$-threshold functions and proper pairs of segments.

\begin{claim} [\cite{Part1}] \label{cl:proper2proper}
	Every proper pair of segments defines a proper $2$-threshold function.
\end{claim}

\begin{theorem} [\cite{Part1}] \label{th:proper_exists}
For any proper $2$-threshold function $f$ on $\gridMN$ there exists a proper pair of segments in $\gridMN$ that defines $f$.
\end{theorem}

\begin{theorem} [\cite{Part1}] \label{th:proper_unique}
		For any proper $2$-threshold function on $\gridMN$ that contains a true point on the boundary of $\gridMN$ there exists a unique proper pair of segments in $\gridMN$ that defines $f$.
\end{theorem}

From \cref{th:proper_unique} it follows that there is a bijection between proper $2$-threshold functions having a true point on the boundary of the grid and the proper pairs of segments defining such functions. 
In this section we will estimate the number of $2$-threshold functions that \emph{do not} satisfy the conditions of \cref{th:proper_unique} and the number of proper pairs of segments that define those functions. We will then use these estimates to express 
the number of $2$-threshold functions via the number of proper pairs of segments.
Then, in the subsequent sections we will use this relation to derive that the number of $2$-threshold functions is asymptotically equal to the number of proper pairs of segments.


%

	\begin{claim}\label{cl:distance-1}
	Let $\{\dir{AB},\dir{CD}\}$ be a \superb pair of segments in $\gridMN$, and let $f$ be the 
	$2$-threshold function defined by $\{\dir{AB},\dir{CD}\}$. 
	If $f$ does not have true points on the boundary of the grid, i.e. 
	$M_1(f) \subseteq \{1,\dots,m-2\}\times\{1,\dots,n-2\}$, then the distances $d(A,\ell(CD))$ and $d(B,\ell(CD))$ do not exceed one.
	\end{claim}
	\begin{proof}
	The statement is obvious for $\ell(AB) = \ell(CD)$, so we assume that $AB$ and $CD$ are not collinear.

	Let us first assume that $\ell(AB)$ and $\ell(CD)$ are not parallel and denote by $O$ the intersection point of the two lines.
	We start by showing that there exists a point $X \in \ell(AB) \cap B(\gridMN)$ such that $AB \subseteq OX$.
	Indeed, since $f(A) = 1$, the point $A$ is an interior point of $\Conv(\gridMN)$, and hence the line $\ell(AB)$ intersects
	$B(\gridMN)$ in exactly two points, which we denote by $X$ and $Y$.
	Furthermore, as $\ell(CD)$ does not separate $A$ and $B$, we have either $AB \subseteq OX$ or $AB \subseteq OY$.
	Without loss of generality assume $AB \subseteq OX$.
	Let $Z \in B(\gridMN)$ be the closest point to $X$ such that $f_{\dir{AB}}(Z) = 1$. Clearly, $d(X,Z) \leq 1$.
	The assumption $M_1(f) \subseteq \{1,\dots,m-2\}\times\{1,\dots,n-2\}$ implies that $f(Z) = 0$, and therefore $f_{\dir{CD}}(Z) = 0$.
	Hence, either $Z \in \ell(CD)$ or the triangle $\ortr{C}{D}{Z}$ is clockwise.
	The former implies that $d(X, \ell(CD)) \leq 1$.
	The latter leads to the same conclusion, if we notice that the triangle $\ortr{C}{D}{X}$ is counterclockwise as $X$ and $A$ 
	lie on the same side of $\ell(CD)$, and hence $\ell(CD)$ intersects $XZ$.
	Finally, since $A,B \in OX$, we conclude that $\max\{d(A,\ell(CD)),d(B,\ell(CD))\} \leq d(X,\ell(CD)) \leq 1$, as required.
	
	The proof for parallel $\ell(AB)$ and $\ell(CD)$ is similar and uses the fact that the distance from any point of 
	$\ell(AB)$ to $\ell(CD)$ is the same. 
	\end{proof}

	\begin{claim}\label{cl:no-ones-on-grid}
	There are $O(m^2n^2(m+n)^2)$ \superb pairs of segments $\{\dir{AB}, \dir{CD}\}$ in $\gridMN$ such that
	the 2-threshold function defined by $\{\dir{AB}, \dir{CD}\}$ does not have true points on the boundary of $\gridMN$.
	\end{claim}
	
	\begin{proof}
	There are at most $mn$ ways to choose each of $C$ and $D$.
	Given the segment $CD$, by \cref{cl:distance-1}, each of $A$ and $B$ lies at distance at most one from 
	$\ell(CD)$. Since there are $O(m+n)$ such points, we conclude that there are $O(m^2n^2(m+n)^2)$ desired pairs of segments.
	\end{proof}

\begin{claim}\label{claim:non-proper-functions}
There are $O(m^2n^2(m+n)^2)$ $2$-threshold functions on $\grid_{m,n}$ that are either threshold or do not have true points on the boundary of $\grid_{m,n}$.
\end{claim}
\begin{proof}
Since the number of threshold functions is $\Theta(m^2n^2)$, it is enough to prove that the number of proper $2$-threshold functions that do not have a true point on the boundary of $\grid_{m,n}$ is $O(m^2n^2(m+n)^2)$.
Indeed, from \cref{th:proper_exists} it follows that every proper $2$-threshold function with no true points on the boundary of $\grid_{m,n}$ is defined by at least one proper pair of segments, and hence the number of such functions can not exceed the number of proper pairs of segments defining these functions, which is estimated in \cref{cl:no-ones-on-grid} as $O(m^2n^2(m+n)^2)$.
\end{proof}

We are now in a position to state formally the main result of the section.
Denote by $q(m,n)$ the number of \superb pairs of segments in $\gridMN$.

\begin{theorem}
\label{th:t_2=p}
	\begin{equation}\label{eq:t2ToAllSuperb}
		t_2(m,n) = q(m,n) + O\big(m^2n^2(m+n)^2\big).
	\end{equation}
\end{theorem}
\begin{proof}
	Denote by $t_{2}'(m,n)$ the number of proper $2$-threshold functions on $\gridMN$ with a true point on the boundary of the grid.
	To prove (\ref{eq:t2ToAllSuperb}), we will express both the number of $2$-threshold functions and the number of proper pairs of segments via $t_{2}'(m,n)$.
	
	We start with $2$-threshold functions.
	The set of 2-threshold functions is the disjoint union of three subsets:
	\begin{enumerate}
		\item threshold functions;
		\item proper $2$-threshold functions with no true point on the boundary of the grid;
		\item proper $2$-threshold functions having a true ponit on the boundary of the grid.
	\end{enumerate}
	Hence, by \cref{claim:non-proper-functions}, we have
	\begin{equation} \label{eq:t_2mn}
		t_2(m,n) = t_{2}'(m,n) + O(m^2n^2(m+n)^2).
	\end{equation}

	Now, to express the number of proper pairs of segments we recall that
	by \cref{cl:proper2proper} a proper pair of segments defines a proper 2-threshold function. Hence the set of proper pairs of segments can be partitioned into two subsets:
	\begin{enumerate}
		\item proper pairs of segments defining proper $2$-threshold functions with no true points on the boundary of the grid
		\item proper paris of segments defining proper $2$-threshold functions having a true ponit on the boundary of the grid
		\end{enumerate}
	Note that the number of pairs of segments in the second set is equal to $t_{2}'(m,n)$ due to 
	\cref{th:proper_unique}. Therefore, by \cref{cl:no-ones-on-grid}, we have
	$$
	q(m, n) = t_{2}'(m,n) + O(m^2n^2(m+n)^2).
	$$
	The above formula together with (\ref{eq:t_2mn}) imply the theorem.
\end{proof}

\cref{th:t_2=p} is useful if the order of the number of $2$-threshold functions (and proper pairs of segments) is larger than $m^2n^2(m+n)^2$.
It follows from (\ref{eq:k_trivial_upper_bound}) that $t_2(m,n) = O(m^4n^4)$. 
In the subsequent sections we will prove that $q(m,n) = \Theta(m^4n^4)$.




\section{Pairs of prime segments in convex position}
\label{sec:from2thresholdToProperPairs}

In this section we will reduce the estimation of the number of proper pairs of oriented segments to that of prime (non-oriented) segments in convex position.
First we will show that the number of those proper pairs of segments, which are not in convex position, does not affect the asymptotics.
To this end, we will use the structure of proper pairs of segments revealed in \cite{Part1}:

\begin{theorem}[\cite{Part1}]
\label{th:superb_rectangle}
The pair of prime segments $\dir{AB},\dir{CD}$ is \superb if and only if one of the following holds:
\begin{enumerate}
\item[(1)] $AC \subset BD$;
\item[(2)] $A \in BD$ and $\ortr{C}{D}{B}$ is a counterclockwise triangle or $C \in BD$ and $\ortr{A}{B}{D}$ is a counterclockwise triangle;
\item[(3)] $\dir{ABCD}$ is a counterclockwise quadrilateral.
\end{enumerate}
\end{theorem}

Let $p(n,m)$ denote the number of proper pairs of segments, which are in convex position.
Then we have the following
\begin{claim}
\label{cl:q_p}
\begin{equation}\label{eq:AllSuperbToSqSuperb}
		q(m,n) = p(m,n) + O\big(m^3n^3(m+n)\big).
\end{equation}	
\end{claim}
\begin{proof}
We will show that the number of \superb pairs of segments $\{\dir{AB},\dir{CD}\}$ in
	$\gridMN$ such that $\Conv(\{A,B,C,D\})$ is a segment or triangle is $O(m^3n^3(m+n))$.
	If $\Conv(\{A,B,C,D\})$ is a segment, then all of the four points $A,B,C,$ and $D$ belong to the same line. 
	If $\Conv(\{A,B,C,D\})$ is a triangle, then, by \cref{th:superb_rectangle}, three of the points belong to the same line.
	In both cases there are three collinear points, say $A,B,C$. There are $O(m^2n^2)$ ways to choose two of these three points.
	Given two fixed points, there are at most $\max\{m-2,n-2\} = O(m+n)$ ways to choose the third one. 
	For the fourth point, whether it lies on the same line with $A,B,C$ or not, there are $O(mn)$ choices.
	Hence, altogether there are $O(m^3n^3(m+n))$ \superb pairs of segments $\{\dir{AB},\dir{CD}\}$ in $\gridMN$ such that 
	$\Conv(\{A,B,C,D\})$ is a segment or triangle, which implies (\ref{eq:AllSuperbToSqSuperb}).
\end{proof}

\cref{th:t_2=p} and \cref{cl:q_p} reduce the estimation of the number of $2$-threshold functions to that of the number of proper pairs of segments in convex position.
Further, we will show that it suffices to consider non-oriented prime segments in convex position.
The following claim is a convenient necessary and sufficient condition for a pair of segments to be in convex position.

\begin{claim}
\label{cl:proper_criteria}
Segments $AB$ and $CD$ are in convex position if and only if
\begin{equation}\label{eq:proper_pair}
\begin{cases}
\ell(AB) \cap CD = \emptyset,\\
\ell(CD) \cap AB = \emptyset.
\end{cases}
\end{equation}
\end{claim}
\begin{proof}
Clearly, if $AB$ and $CD$ are in convex position, then (\ref{eq:proper_pair}) holds.
To prove the converse, we observe that (\ref{eq:proper_pair}) implies that $\Conv(\{A,B,C,D\})$ is not a segment or triangle, hence it is a convex quadrilateral with vertices $A,B,C,$ and $D$.
Moreover, $AB \cap CD = \emptyset$, and hence the segments are neither diagonals nor adjacent edges, and consequently they are opposite edges of the quadrilateral $\Conv(\{A,B,C,D\})$.
\end{proof}


The relation between \superb pairs of segments in convex position and pairs of non-oriented prime segments in convex position is revealed in the following theorem.

\begin{figure}
\begin{subfigure}[b]{0.45\textwidth}
\begin{tikzpicture}[scale=1.5]
	\foreach \x in {0,...,3}
    	\foreach \y in {0,...,2}
    	{
    		\node[grid_point] (\x\y) at (\x,\y) {};
    	}

    	
    	\node[false_point] (D) at (1,2) {};   
    	\node[true_point] (A) at (1,1) {};   
    	\node[false_point] (B) at (2,0) {};   
    	\node[true_point] (C) at (3,1) {};   
    	
    	\draw[oriented_segment] (C) -- (D);
    	\draw[oriented_segment] (A) -- (B);

 		\node[small_text, black, left] at (1,1) {$A$};
 		\node[small_text, black, below right] at (2,0) {$B$};
 		\node[small_text, black, below right] at (3,1) {$C$};
 		\node[small_text, black, left] at (1,2) {$D$};
\end{tikzpicture}
\end{subfigure}\hfill
\begin{subfigure}[b]{0.45\textwidth}
\begin{tikzpicture}[scale=1.5]
	\foreach \x in {0,...,3}
    	\foreach \y in {0,...,2}
    	{
    		\node[grid_point] (\x\y) at (\x,\y) {};
    	}

    	\node[false_point] (D) at (1,2) {};   
    	\node[false_point] (A) at (1,1) {};   
    	\node[false_point] (B) at (2,0) {};   
    	\node[false_point] (C) at (3,1) {};   
    	
    	\draw[black] (D) -- (C);
    	\draw[black] (A) -- (B);

 		\node[small_text, black, left] at (A) {$A$};
 		\node[small_text, black, below right] at (B) {$B$};
 		\node[small_text, black, below right] at (C) {$C$};
 		\node[small_text, black, left] at (D) {$D$};
\end{tikzpicture}
\end{subfigure}
\caption{The proper pair of segments $\{\protect\dir{AB},\protect\dir{CD}\}$ in convex position and the corresponding pair of prime segments $AB, CD$ in convex position.}
\label{fig:convex_position}
\end{figure}
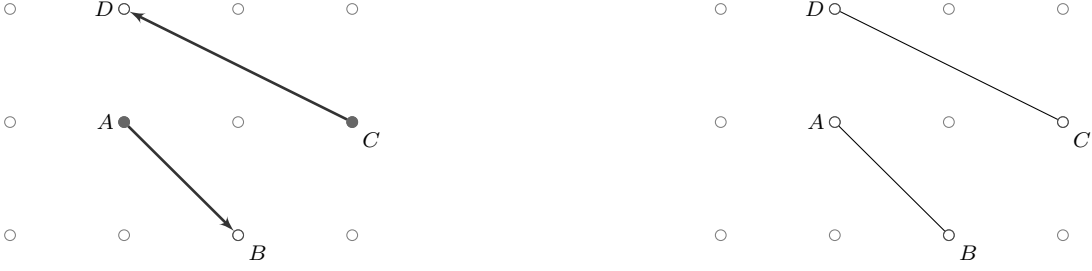

\begin{theorem}\label{prop:superbToProperPairs}
There is one-to-one correspondence between pairs of \linebreak (non-oriented) prime segments in convex position and \superb pairs of segments in convex position.
\end{theorem}
\begin{proof}
To prove the claim, we establish a bijective mapping between the two sets of pairs of segments.
Clearly, if $\{\dir{AB},\dir{CD}\}$ is a \superb pair of segments in convex position, then $AB$ and $CD$ are prime and in convex position. 

Now, let $AB$ and $CD$ be prime segments in convex position, then \linebreak $\Conv(\{A,B,C,D\})$ is a quadrilateral, and $AB$ and $CD$ are two of its four edges.
Assume, without loss of generality, that
$\Conv(\{A,B,C,D\})$ has also the edges $CB$ and $DA$ (see \cref{fig:convex_position}).
To apply \cref{th:superb_rectangle} we are to consider both oriented quadrilaterals corresponding to $\Conv(\{A,B,C,D\})$, namely, $\dir{ABCD}$ and $\dir{DCBA}$. 
These quadrilaterals have opposite orientations, we assume, without loss of generality, that $\dir{ABCD}$ is the counterclockwise one. 
Then from \cref{th:superb_rectangle} it follows that $\{\dir{AB},\dir{CD}\}$ is a unique proper pair of segments in convex position corresponding to $AB, CD$.

\end{proof}


Because of the bijection established in \cref{prop:superbToProperPairs}, $p(m,n)$ denotes both the number of \superb pairs of segments in convex position and the number of pairs of (non-oriented) prime segments in convex position.
By \cref{th:t_2=p}, \cref{cl:q_p}, and \cref{prop:superbToProperPairs}, we conclude that the number of $2$-threshold functions can be expressed via the number of pairs of prime segments in convex position:

\begin{corollary}
\label{cor:t_2-p}
$$
t_2(m,n) = p(m,n) + o(m^4n^4).
$$
\end{corollary}

As we will show in \cref{th:s_general_formula}, the number of pairs of prime segments in convex position is $\Theta(m^4n^4)$, and hence, can be used to derive an asymptotic formula for the number of $2$-threshold functions.

\section{The number of pairs of prime segments in convex position}
\label{sec:numberOfProperPairs}
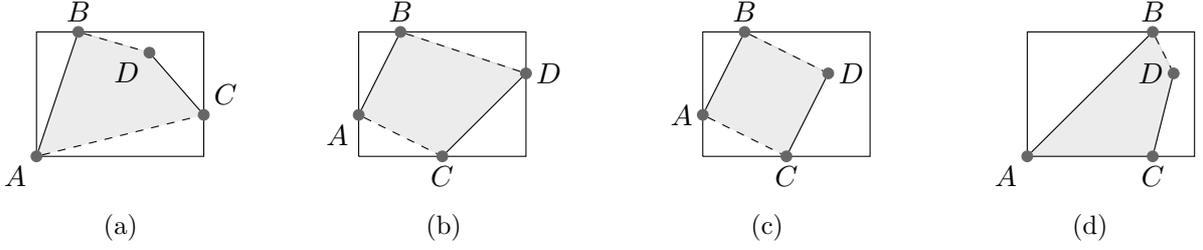
\begin{figure}
\begin{subfigure}[t]{0.25\textwidth}
\centering
\begin{tikzpicture}[scale=0.55]
	
	\draw[light_grey,fill=light_grey!50] (0,0) -- (1,3) -- (2.7,2.5) -- (4,1) ;
	\draw[black] (0,0) -- (4,0) -- (4,3) -- (0,3) -- (0,0);	
	\node[true_point] (A) at (0,0) {};
	\node[true_point] (B) at (1,3) {};
	\node[true_point] (C) at (4,1) {};
	\node[true_point] (D) at (2.7,2.5) {};
	\draw[black] (A) -- (B);
	\draw[black] (C) -- (D);
	\draw (A) node [below left] {$A$};
	\draw (B) node [above] {$B$};
	\draw (C) node [above right] {$C$};
	\draw (D) node [below left] {$D$};
	\draw[dashed] (A) -- (C);
	\draw[dashed] (B) -- (D);
	
\end{tikzpicture}
\caption{}
\label{fig:inscribed-1}
\end{subfigure}\hfill\hfill
\begin{subfigure}[t]{0.25\textwidth}
\centering
\begin{tikzpicture}[scale=0.55]
	
	\draw[light_grey,fill=light_grey!50] (0,1) -- (1,3) -- (4,2) -- (2,0);
	\draw[black] (0,0) -- (4,0) -- (4,3) -- (0,3) -- (0,0);	
	\node[true_point] (A) at (0,1) {};
	\node[true_point] (B) at (1,3) {};
	\node[true_point] (C) at (2,0) {};
	\node[true_point] (D) at (4,2) {};
	\draw[black] (A) -- (B);
	\draw[black] (C) -- (D);
	\draw (A) node [below left] {$A$};
	\draw (B) node [above] {$B$};
	\draw (C) node [below] {$C$};
	\draw (D) node [right] {$D$};
	\draw[dashed] (A) -- (C);
	\draw[dashed] (B) -- (D);
	
\end{tikzpicture}
\caption{}
\label{fig:inscribed-2}
\end{subfigure}\hfill\hfill
\begin{subfigure}[t]{0.25\textwidth}
\centering
\begin{tikzpicture}[scale=0.55]
	
	\draw[light_grey,fill=light_grey!50] (0,1) -- (1,3) -- (3,2) -- (2,0) ;
	\draw[black] (0,0) -- (4,0) -- (4,3) -- (0,3) -- (0,0);	
	\node[true_point] (A) at (0,1) {};
	\node[true_point] (B) at (1,3) {};
	\node[true_point] (C) at (2,0) {};
	\node[true_point] (D) at (3,2) {};
	\draw[black] (A) -- (B);
	\draw[black] (C) -- (D);
	\draw (A) node [left] {$A$};
	\draw (B) node [above] {$B$};
	\draw (C) node [below] {$C$};
	\draw (D) node [right] {$D$};
	\draw[dashed] (A) -- (C);
	\draw[dashed] (B) -- (D);
	
\end{tikzpicture}
\caption{}
\label{fig:inscribed-3}
\end{subfigure}\hfill\hfill
\begin{subfigure}[t]{0.25\textwidth}
\centering
\begin{tikzpicture}[scale=0.55]
	
	\draw[light_grey,fill=light_grey!50] (0,0) -- (3,3) -- (3.5,2) -- (3,0);
	\draw[black] (0,0) -- (4,0) -- (4,3) -- (0,3) -- (0,0);	
	\node[true_point] (A) at (0,0) {};
	\node[true_point] (B) at (3,3) {};
	\node[true_point] (C) at (3,0) {};
	\node[true_point] (D) at (3.5,2) {};
	\draw[black] (A) -- (B);
	\draw[black] (C) -- (D);
	\draw (A) node [below left] {$A$};
	\draw (B) node [above] {$B$};
	\draw (C) node [below] {$C$};
	\draw (D) node [left] {$D$};
	\draw[dashed] (B) -- (D);
	
\end{tikzpicture}
\caption{}
\label{fig:inscribed-4}
\end{subfigure}\hfill\hfill
\caption{$\{AB, CD\}$ is a pair of segments in convex position.
The grey shape is $\protect\Conv(\{A,B,C,D\})$. 
The rectangle is circumscribed about $\protect\Conv(\{A,B,C,D\})$ in (a) and (b) and not circumscribed in (c) and (d).}
\label{fig:inscribed}
\end{figure}


In what follows we will extensively use rectangles with horizontal and vertical sides  circumscribed about the convex quadrilaterals being the convex hulls of pairs of segments in convex position (see \cref{fig:inscribed}).

Denote by $\rMN$ a $\m \times \n$ rectangle $\Conv(\grid_{\m+1,\n+1})$ for natural numbers $u$ and $v$.
Denote by $Z(u,v)$ the set of pairs of prime segments $\{AB, CD\}$ in convex position such that $\rMN$ is circumscribed about $\Conv(\{A,B,C,D\})$.

\begin{theorem}
	\begin{equation}\label{eq:total_1}
	p(m,n) = \sum_{\m=1}^m \sum_{\n=1}^n (m-\m)(n-\n)|Z(\m,\n)|.
	\end{equation}
\end{theorem}
\begin{proof}
	First for every convex quadrilateral with vertices in $\gridMN$ there exists a unique rectangle with sides parallel to the sides of $\Conv(\gridMN)$ circumscribed about it.
	Hence, the statement follows from the fact that there are exactly $(m-\m)(n-\n)$ rectangles in $\gridMN$ 
	with sides of length $\m$ and $\n$ that are parallel to the sides of $\Conv(\gridMN)$.
\end{proof}

Let $Z_i(\m,\n) \subseteq Z(\m,\n)$ be the set of those pairs of segments $AB$, $CD$ in $Z(\m,\n)$, for which exactly $i$ points in $\{A,B,C,D\}$ are vertices of $\rMN$. 
Clearly, $Z(\m,\n)$ is the disjoint union of $Z_i(\m,\n), i = 0,1, \ldots, 4$, and therefore
\begin{equation}\label{eq:sumZ}
	|Z(\m,\n)| = \sum_{i=0}^{4}|Z_i(\m,\n)|.
\end{equation}

Our next step is to estimate the cardinality of $Z_i(\m,\n)$ for every $i \in \{0,1, \ldots, 4\}$.
The cases $i \in \{4, 3\}$ are easy and we consider them below. The cases $i \in \{ 0,1,2 \}$ are more involved
and we treat them independently in Sections \ref{sec:2points}--\ref{sec:0points}.

\begin{lemma}\label{lem:Z34}
	$|Z_3(\m,\n)| + |Z_4(\m,\n)| = O(\m\n).$
\end{lemma}
\begin{proof}
	By definition, for any pair of segments $\{ AB, CD \} \in Z_3(\m,\n) \cup Z_4(\m,\n)$ at least three of the endpoints
	of the segments are vertices of $\rMN$.
	Therefore, since there is a constant number of ways to map 3 of the endpoints of the segments
	to the vertices of $\rMN$, and there are $O(\m\n)$ ways to place the fourth point in $\grid_{u+1,v+1}$, we conclude the lemma.
\end{proof}

\subsection{The number of pairs of segments with two corner points}
\label{sec:2points}

In this section we estimate $|Z_{2}(\m,\n)|$, i.e. the number of pairs of segments $\{ AB, CD \}$ in $Z(\m,\n)$, for which
exactly $2$ points in $\{A,B,C,D\}$ are vertices of $\rMN$. 

Let $\{X, Y\} = \{A,B,C,D\} \cap \text{Vert}(\rMN)$. 
We consider the partition of $Z_{2}(\m,\n)$ into the following three subsets:
\begin{enumerate}
	\item $Z_{2}^a(\m,\n)$ is the subset of $Z_{2}(\m,\n)$ such that $X$ and $Y$ are adjacent vertices of $\rMN$, i.e. $XY$ is a side of $\rMN$.

	\item $Z_{2}^b(\m,\n)$ is the subset of $Z_{2}(\m,\n)$ such that $X$ and $Y$ are opposite vertices of $\rMN$ and 
	belong to the same segment.

	\item $Z_{2}^c(\m,\n)$ is the subset of $Z_{2}(\m,\n)$ such that $X$ and $Y$ are opposite vertices of $\rMN$ and 
	belong to the different segments.
\end{enumerate}

\noindent
Clearly,
$$
|Z_2(\m,\n)| = |Z_{2}^a(\m,\n)| + |Z_{2}^b(\m,\n)| + |Z_{2}^c(\m,\n)|.
$$
Let us show that the first summand does not affect the asymptotics of the sum which will be proved to be $\Theta(\m^2\n^2)$.

\begin{lemma}\label{lem:Z2a}
$
|Z_{2}^a(\m,\n)| = O(\m^2\n + \m\n^2).
$
\end{lemma}
\begin{proof}
Since $\rMN$ is circumscribed about $\Conv(AB \cup CD)$ and two of the points $A,B,C,D$ belong to the same side of $\rMN$, at least one of the other two points belongs to the opposite side of $\rMN$.
Therefore there are $O(\m+\n)$ ways to place this point.
Furthermore, there are $O(\m\n)$ ways to place the fourth point in $\rMN$, which implies the desired estimate.
\end{proof}

\begin{lemma}\label{prop:P+-}
Let $AB$ and $CD$ be segments with endpoints in $\gridMN$.
Then $AB$ and $CD$ are in convex position if and only if $A,B,C,$ and $D$ are in general position, the triangle $\ortr{A}{B}{D}$ has the same orientation as $\ortr{A}{B}{C}$, and the triangle $\ortr{C}{D}{A}$ has the same orientation as $\ortr{C}{D}{B}$.
\end{lemma}
\begin{proof}
We will prove the lemma by showing that its conditions are equivalent to those of \cref{cl:proper_criteria}. 
First we claim that the equation $\ell(AB) \cap CD = \emptyset$ is equivalent to the statement that the points in both sets $\{A,B,C\}$ and $\{A,B,D\}$ are in general position and the orientations of
$\ortr{A}{B}{C}$ and $\ortr{A}{B}{D}$ are the same.
Indeed, $\ortr{A}{B}{C}$ and $\ortr{A}{B}{D}$ are triangles if and only if $C,D \notin \ell(AB)$.
Moreover, the orientations of $\ortr{A}{B}{C}$ and $\ortr{A}{B}{D}$ are the same if and only if $\ell(AB)$ does not separate $C$ and $D$.

Using similar arguments, one can establish the equivalence of the equation \linebreak $\ell(CD) \cap AB = \emptyset$ and the statement that the points in both sets $\{C,D,A\}$ and $\{C,D,B\}$ are in general position and the orientations of triangles $\ortr{C}{D}{A}$ and $\ortr{C}{D}{B}$ are the same.
\end{proof}

Let $A,B,C$ be three points in $\gridMN$ in general position.
We say that $\mathcal{P} \subset \Conv(\gridMN)$ is the \emph{admissible region
with respect to} $A,B,C$ if for every point $D \in \mathcal{P}$
the segments $AB$ and $CD$ are in convex position.
%
\cref{prop:P+-} implies the following description of the admissible regions.

\begin{corollary}\label{cor:P+-}
	Let $A,B,C$ be three points in $\gridMN$ in general position.
	Then the interior of $\mathcal{P}_1 \cup \mathcal{P}_2$  
	defined as on \cref{fig:P+} is the admissible region with respect to $A,B,C$.
\end{corollary}


\begin{figure}
\centering
\begin{tikzpicture}[scale=1.2]
  
    \draw[light_grey,fill=light_grey!50] (-2,0) -- (0,0) -- (3,2) -- (2,3) -- (1.5,2.9) -- (1, 2.8) -- (0, 2.4) -- (-1, 1.7);
    \draw[light_grey,fill=light_grey!50] (7,0) -- (5,0) -- (3,2) -- (5,10/3) -- (5.5,9/3) -- (6, 7.5/3) -- (6.5, 4.5/3);
    \pattern[pattern=north east hatch, hatch distance=2.5mm, pattern color = gray, hatch thickness=.5pt] (-2,0) -- (7,0) -- (6.5, 4.5/3) -- (6, 7.5/3) -- (5.5,9/3) -- (5,10/3) -- (1.5,2.9) -- (1, 2.8) -- (0, 2.4) -- (-1, 1.7);
 	\draw (-2,0) -- (7,0);
 	\draw (0,0) -- (5,10/3);
 	\draw (5,0) -- (2,3);
 	\draw[black,fill=white] (0,0) circle (1.4pt);
	\draw (0,0) node [below] {$A$};
	\draw[black,fill=white] (5,0) circle (1.4pt);
	\draw (5,0) node [below] {$B$};
	\draw[black,fill=white] (3,2) circle (1.4pt);
	\draw (3,2.2) node [above] {$C$};
	\draw (1,2) node [above] {$\mathcal{P}_1$};
	\draw (5,2) node [above] {$\mathcal{P}_2$};
	
\end{tikzpicture}
\caption{The stripped region is the area of points $X$ such that the triangle $\protect\ortr{A}{B}{X}$ has the same orientation as $\protect\ortr{A}{B}{C}$. 
The region $\protect\mathcal{P}_1$ is the set of points $X$ such that $\protect\ortr{C}{X}{A}$, $\protect\ortr{C}{X}{B}$ are counterclockwise. 
The region $\protect\mathcal{P}_2$ is the set of points $X$ such that $\protect\ortr{C}{X}{A}$ and $\protect\ortr{C}{X}{B}$ are clockwise.
The union $\protect\mathcal{P}_1 \cup \mathcal{P}_2$ is the admissible region 
with respect to $A,B,C$.}
\label{fig:P+}
\end{figure}
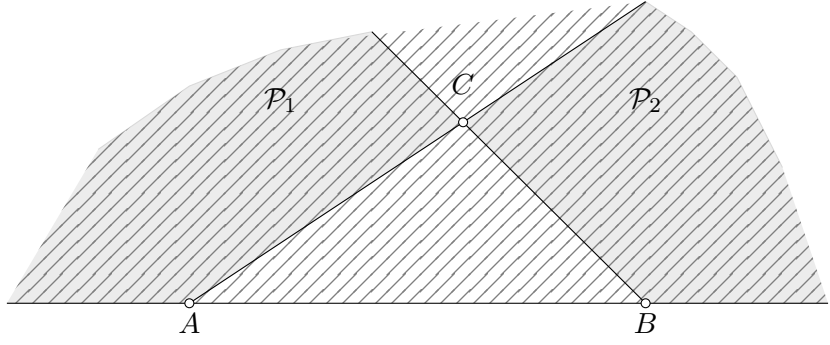

For a polygon $\mathcal{P}$, denote by ${\cal L}(\mathcal{P})$ the number of integer points in $\mathcal{P}$, i.e.
$$
{\cal L}(P) = |\mathbb{Z}^2 \cap \mathcal{P}|.
$$

For a polygon $\mathcal{P}$ and a point $A$ denote by $Prime(\mathcal{P}, A)$ the number of integer points $X \in \mathcal{P}$ such that $AX$ is a prime segment.
If $A$ is the origin $O = (0,0)$ we simply write $Prime(\mathcal{P})$.

\begin{lemma}
\label{lem:S_triangle}
Let $\rMN$ be circumscribed about a triangle $\tr{A}{B}{C}$.
Then
\begin{equation}
\label{eq:Prime_points_in_triangle}
	Prime(\tr{A}{B}{C}, A) = \frac{6}{\pi^2}\Area(\tr{A}{B}{C}) + O(\m+\n).
\end{equation}
\end{lemma}
\begin{proof}
Without loss of generality we assume that $A$ coincides with the origin $O = (0,0)$.
Denote by $i\cdot \tr{A}{B}{C}$ the triangle $\tr{A}{B}{C}$ scaled for a given factor $i>0,$ i.e.
$$
i\cdot \tr{A}{B}{C} \ = \ \left\{(i\cdot x, i\cdot y) \in \mathbb{Z}^2 | (x,y) \in \tr{A}{B}{C} \ \right\}.
$$
We start with the equation
\begin{eqnarray}
\nonumber{\cal L}(\tr{A}{B}{C}) &=& Prime\left(\tr{A}{B}{C}\right)
+ Prime\left(\frac{1}{2}\cdot \tr{A}{B}{C}\right)
+ Prime\left(\frac{1}{3}\cdot \tr{A}{B}{C}\right) + \ldots \\[0.3cm]
\nonumber&=& \sum_{j=1}^{\m+\n} Prime\left(\frac{1}{j}\cdot \tr{A}{B}{C}\right)
\label{first} 
\end{eqnarray}
and the consequent one 
\begin{eqnarray}
\nonumber{\cal L}\left(\frac{1}{p}\cdot \tr{A}{B}{C}\right)  
&=& \sum_{q=1}^{\infty}  Prime\left(\frac{1}{p\cdot q}\cdot \tr{A}{B}{C}\right) \\
\nonumber&=& \sum_{q=1}^{(\m+\n)/p} Prime\left(\frac{1}{p\cdot q}\cdot \tr{A}{B}{C}\right).
\label{second} 
\end{eqnarray}
Using (\ref{eq:mobius}) we proceed with

\begin{eqnarray}
\nonumber Prime(\tr{A}{B}{C}) &=& \sum_{l=1}^{\m+\n}Prime\left( \frac{1}{l}\cdot \tr{A}{B}{C}\right)
\left(\sum_{k | l} \mu(k)\right) \quad \\
\nonumber& = & 
\sum_{h=1}^{\m+\n} \mu(h) 
     \sum_{i=1}^{(\m+\n)/h} Prime\left(\frac{1}{h i}\cdot \tr{A}{B}{C}\right) \\
\nonumber& = & \sum_{h=1}^{\m+\n} \mu(h) {\cal L}\left(\frac{1}{h}\cdot \tr{A}{B}{C}\right).
\end{eqnarray}
From \cite{Davenport1951} we have
$$
 {\cal L}\left(\tr{A}{B}{C}\right) = \Area(\tr{A}{B}{C}) + O\left(\m+\n\right),
$$
and hence
\begin{eqnarray}
\nonumber\sum_{h=1}^{\m+\n} \mu(h) {\cal L}\left(\frac{1}{h} \cdot \tr{A}{B}{C}\right) & = & \sum_{h=1}^{\m+\n} \mu(h) \frac{1}{h^2}\left(\Area(\tr{A}{B}{C}) + 
O\left(\m+\n\right)\right)  \\
 \nonumber \\
\nonumber&=&  \Area(\tr{A}{B}{C}) 
\left(\sum_{h=1}^{\infty} \frac{\mu(h)}{h^2}  - 
      \sum_{h=\m+\n+1}^{\infty} \frac{\mu(h)}{h^2}\right) +  
O\left(\m+\n\right)  \\
\nonumber&\overset{\mathrm{(\ref{eq:mobius_riemann})}}{=}&  \frac{6}{\pi^2} \Area(\tr{A}{B}{C}) +
\nonumber O\left(\frac{\Area(\tr{A}{B}{C})}{\m+\n} \ + \ \m+\n \right)  \\
\nonumber& = & \frac{6}{\pi^2} \Area(\tr{A}{B}{C})  \ + \ O(\m + \n).
\end{eqnarray}
\end{proof}


\begin{corollary}
\label{cor:S-triangle}
Let $\rMN$ be circumscribed about a triangle $\tr{A}{B}{C}$.
Then the number of internal points $X$ of $\tr{A}{B}{C}$ such that $AX$ is a prime segment is
\begin{equation}
	\frac{6}{\pi^2}\Area(\tr{A}{B}{C}) + O(\m+\n).
\end{equation}
\end{corollary}

The lemma and the corollary above will be applied in the rest of the section and in \cref{sec:one_corner_point} in the following way. 
\cref{cor:P+-} implies that given the points $A,B,C$ all the possible choices for $D$ such that $AB$ and $CD$ are in convex position are contained in the admissible region $\mathcal{P}$.
In the most cases $\mathcal{P}$ will be a pair of triangles that have a common vertex $C$.
In both cases we will use \cref{lem:S_triangle} (or its corollary) to estimate the number of possible points $D$ in $\mathcal{P}$ such that $CD$ is a prime segment.

\begin{lemma}\label{lem:Z2b}
\begin{equation*}
\label{eq:4_final}
|Z_2^b(\m,\n)| = \begin{cases}
\frac{1}{\pi^2}\m^2\n^2 + O(\m^2 \n + \m \n^2) & \text{if $\m\perp \n$},\\
0 & \text{otherwise}.
\end{cases}
\end{equation*}
\end{lemma}
\begin{proof}

\begin{figure}
\centering
\begin{tikzpicture}[scale=1.2]

	\draw [<->,thick] (0,4) node (yaxis) [above] {}
       |- (6,0) node (xaxis) [right] {};
				
	\draw (0,0) -- (5,0);
	\draw (0,0) -- (0,3);
	\draw (0,3) -- (5,3);
	\draw (5,0) -- (5,3);
	\draw (5,0) -- (0,3);
	\draw[black,fill=light_grey] (0,5/3) -- (2,1) -- (0,3);
	\draw[black,fill=light_grey] (3,0) -- (2,1) -- (5,0);
	\draw[black,fill=white] (0,0) circle (1.4pt);
	\draw (0,0) node [left] {$R_1$};
	\draw (5,0) -- (0,5/3);
	\draw[black,fill=white] (0,5/3) circle (1.4pt);
	\draw (0,5/3) node [left] {$Y$};
	\draw (0,3) -- (3,0);
	\draw[black,fill=white] (3,0) circle (1.4pt);
	\draw (3,0) node [below] {$X$};
	\draw[black,fill=white] (2,1) circle (1.4pt);
	\draw (2,1) node [above right] {$C$};
	\draw[black,fill=white] (0,3) circle (1.4pt);
	\draw (0,3) node [left] {$A=R_2$};
	\draw[black,fill=white] (5,3) circle (1.4pt);
	\draw (5,3) node [right] {$R_3$};
	\draw[black,fill=white] (5,0) circle (1.4pt);
	\draw (5,0) node [below] {$B=R_4$};
	
\end{tikzpicture}
\caption{$C \in \tr{A}{B}{R_1}$, the grey triangles form the admissible region 
	with respect to $A,B,C$.}
\label{fig:main_4}
\end{figure}
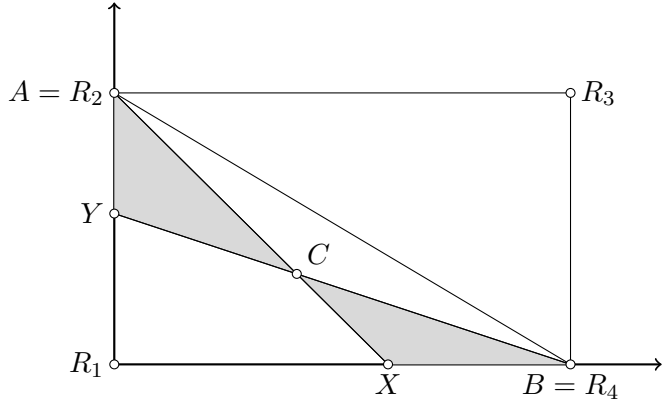

First, we notice that for non-coprime $\m$ and $\n$ the diagonal of $\rMN$ is not a prime segment and the set $Z_2^b(\m,\n)$ is empty.

Let now $\m$ and $\n$ be coprime. Let us denote the vertices of $\rMN$ by $R_1, R_2, R_3$, and $R_4$ as in \cref{fig:main_4}, and
consider a pair $\{AB, CD\}$ from $Z_2^b(\m,\n)$.
Without loss of generality we assume that $AB$ is a diagonal of $\rMN$, i.e. either $AB = R_2R_4$ or $AB = R_1R_3$.
Let us assume that $AB = R_2R_4$. 
Since, by definition, $CD$ does not intersect $AB$, either $CD \in \tr{A}{B}{R_1}$ or $CD \in \tr{A}{B}{R_3}$.
Let us assume that $CD \in \tr{A}{B}{R_1}$ and let $C=(c_1,c_2)$, $D = (d_1,d_2)$.
Clearly, $c_1 \neq d_1$ as otherwise $\ell(CD)$ would intersect $AB$.
Without loss of generality we assume that $c_1 > d_1$. 
Let us denote
$$
X=\ell(AC) \cap R_1B = \left(\frac{\n c_1}{\n-c_2},0\right)
$$ 
and 
$$
Y =  \ell(BC) \cap R_1A = \left(0, \frac{\m c_2}{\m-c_1}\right).
$$
It follows from \cref{cor:P+-} that $AB$ and $CD$ are in convex position if and only if $D$ is an interior point of $ACY$ or $BCX$.
As $c_1 > d_1$ we conclude that $D \in ACY$. 
By \cref{lem:S_triangle}, the number of choices for point $D$ such that $CD$ is a prime segment is 
$$
\frac{6}{\pi^2}\Area(\tr{A}{C}{Y}) + o \left(c_1 \left(\n - c_2\right)\right)= \frac{6}{\pi^2}\Area(\tr{A}{C}{Y}) + O(\m + \n).
$$
Hence, summing up over all possible choices for the point $C$ in $\tr{A}{B}{R_1}$ and multiplying by 4 to take into account
the cases of $C \in \tr{A}{B}{R_3}$ and $AB=R_1R_3$ we derive
\begin{flalign*}
\nonumber |Z_2^b(u,v)| &= 
4\sum_{c_1=1}^{\m-1}\sum_{c_2=1}^{\left\lfloor\frac{\n(\m-c_1)}{\m}\right\rfloor} \left( \frac{6}{\pi^2}\Area(\tr{A}{C}{Y}) + O(\m + \n) \right) \\
&= \frac{24}{\pi^2} \sum_{c_1=1}^{\m-1}\sum_{c_2=1}^{\left\lfloor\frac{\n(\m-c_1)}{\m}\right\rfloor} \Area(\tr{A}{C}{Y}) + O(\m^2 \n + \m \n^2),
\end{flalign*}
where
$$
\Area(\tr{A}{C}{Y}) = \frac{c_1 \cdot d(A,Y)}{2} = \frac{1}{2}c_1 \left(\n - \frac{\m c_2}{\m-c_1}\right)=\frac{1}{2}\left(\n c_1 - \frac{\m c_1c_2}{\m-c_1}\right).
$$
Therefore, we have
\begin{flalign*}
|Z_2^b(u,v)| &= \frac{12}{\pi^2}\sum_{c_1=1}^{\m-1}\sum_{c_2=1}^{\left\lfloor\frac{\n(\m-c_1)}{\m}\right\rfloor}\left(\n c_1 - \frac{\m c_1}{\m-c_1}c_2\right) + O(\m^2 \n + \m \n^2)\\
&\overset{\mathrm{(\ref{eq:sum_powers})}}{=}\frac{12}{\pi^2}\sum_{c_1=1}^{\m-1}\left(\frac{\n^2c_1(\m-c_1)}{\m} - \frac{\m c_1}{\m-c_1}\left(\frac{\n^2(\m-c_1)^2}{2\m^2}+ O\left(\frac{\n(\m-c_1)}{\m}\right)\right)\right) \\
&+ O(\m^2 \n + \m \n^2)\\
&=\frac{12}{\pi^2}\sum_{c_1=1}^{\m-1}\left(\frac{\n^2c_1(\m-c_1)}{2\m}\right) + O(\m^2 \n + \m \n^2)\\
&=\frac{6\n^2}{\pi^2}\sum_{c_1=1}^{\m-1}\left( c_1 - \frac{c_1^2}{\m}\right) + O(\m^2 \n + \m \n^2)\\
&=\frac{6\n^2}{\pi^2}\left(\frac{(\m-1)^2}{2} - \frac{(\m-1)^3}{3\m} + O(\m)\right) + O(\m^2 \n + \m \n^2) \\
&=\frac{1}{\pi^2}\m^2\n^2 + O(\m^2 \n + \m \n^2).
\end{flalign*}

\end{proof}

\begin{lemma}\label{lem:Z2c}
$
|Z_2^c(\m,\n)| = \frac{42}{\pi^4}\m^2\n^2 + o(\m^2\n^2).
$
\end{lemma}
\begin{proof}
Consider a pair $\{AB, CD\}$ from $Z_2^c(\m,\n)$.
Without loss of generality we assume $\{A,B,C,D\} \cap \Vertic(\rMN) = \{A,C\}$, that is
either $\{A,C\} = \{R_1,R_3\}$ or $\{A,C\} = \{R_2,R_4\}$.
The cases are symmetric, and hence it suffices to consider one of them, say $\{A,C\} = \{R_1,R_3\}$.
Without loss of generality we assume $A=R_1$ and $C=R_3$ as in \cref{fig:AC-corners}.
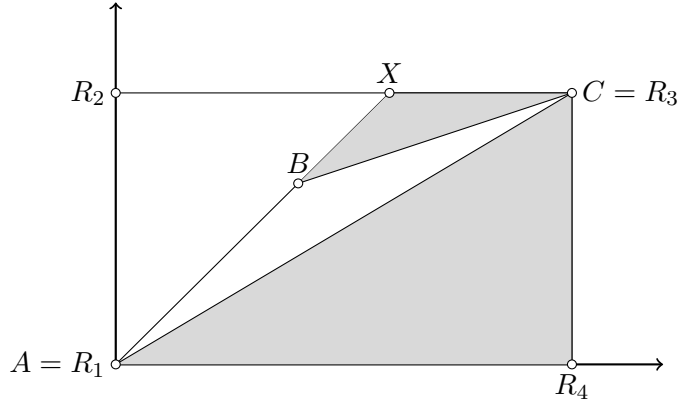
\begin{figure}
\centering
\begin{tikzpicture}[scale=1.2]
	\draw [<->,thick] (0,4) node (yaxis) [above] {}
       |- (6,0) node (xaxis) [right] {};
				
	\draw (0,0) -- (5,0);
	\draw (0,0) -- (0,3);
	\draw (0,3) -- (5,3);
	\draw (5,0) -- (5,3);
	\draw (0,0) -- (3,3);
	\draw[black,fill=light_grey] (0,0) -- (5,3) -- (5,0);
	\draw[black,fill=light_grey] (3,3) -- (5,3) -- (2,2);
	\draw[black,fill=white] (0,0) circle (1.4pt);
	\draw (0,0) node [left] {$A=R_1$};
	\draw[black,fill=white] (3,3) circle (1.4pt);
	\draw (3,3) node [above] {$X$};
	\draw[black,fill=white] (2,2) circle (1.4pt);
	\draw (2,2) node [above] {$B$};
	\draw[black,fill=white] (0,3) circle (1.4pt);
	\draw (0,3) node [left] {$R_2$};
	\draw[black,fill=white] (5,3) circle (1.4pt);
	\draw (5,3) node [right] {$C=R_3$};
	\draw[black,fill=white] (5,0) circle (1.4pt);
	\draw (5,0) node [below] {$R_4$};
	
\end{tikzpicture}
\caption{The point $B$ belongs to $\tr{A}{R_2}{R_3}$, the grey triangles form the admissible region
	with respect to $A,B,C$.}
\label{fig:AC-corners}
\end{figure}
The point $B =(b_1,b_2)$ belongs to one of the triangles $\tr{A}{C}{R_2}$ and $\tr{A}{C}{R_4}$.
Due to symmetry, we assume without loss of generality $B \in \tr{A}{C}{R_2}$, in which case we have~$b_2 > \frac{\n b_1}{\m}$.
Let us denote 
$$
X = \ell(AB) \cap CR_2 = \left(\frac{\n b_1}{b_2}, \n\right).
$$
It follows from \cref{cor:P+-} that $AB$ and $CD$ are in convex position if and only if $D$ is an interior point of 
$\tr{A}{C}{R_4} \cup \tr{B}{C}{X}$ or an interior point of one of the segments $CX$, $AR_4$, or $CR_4$.
By \cref{lem:S_triangle} and \cref{cor:S-triangle}, the number of possible choices for $D$ such that $CD$ is a prime segment for a fixed $B$ is
$$
\frac{6}{\pi^2}\big(\Area(\tr{A}{C}{R_4}) + \Area(\tr{B}{C}{X})\big) + O(\m + \n),
$$
where
$$
\Area(\tr{A}{C}{R_4}) = \frac{\m\n}{2}
$$
and
\begin{flalign*}
\Area(\tr{B}{C}{X}) &= \frac{(\n-b_2) \cdot d(C,X)}{2}= 
\frac{1}{2}(\n-b_2)\left(\m - \frac{\n b_1}{b_2}\right) \\
&= 
\frac{1}{2}\left(uv + vb_1 - ub_2 - \frac{v^2b_1}{b_2} \right).
\end{flalign*}
Therefore, summing over all possible choices of $B$ and multiplying by 4 to take into account the cases 
$B \in \tr{A}{C}{R_4}$ and $\{A,C\} = \{R_2,R_4\}$ we derive:
\begin{flalign}
\label{eq:3.0}
|Z_2^c(\m,\n)| = \nonumber&4\sum_{b_1=1}^{\m-1}\sum_{\substack{b_2=\left\lfloor\frac{\n b_1}{\m}+1\right\rfloor \\ b_2 \perp b_1}}^{\n}\frac{6}{\pi^2}(\Area(\tr{A}{C}{R_4}) + \Area(\tr{B}{C}{X}) + O(\m + \n))\\
&= \frac{12}{\pi^2}\sum_{b_1=1}^{\m-1}\sum_{\substack{b_2=\left\lfloor\frac{\n b_1}{\m}+1\right\rfloor \\ b_2 \perp b_1}}^{\n}\left(2\m\n + \n b_1 - \m b_2 - \frac{\n^2 b_1}{b_2}\right) + O(\m^2\n + \m\n^2).
\end{flalign}
We will estimate the asymptotics of different summands of (\ref{eq:3.0}) separately.

\noindent
\textbf{1. Estimation of $\sum\sum 2\m\n$.}

Using formulas (\ref{eq:sum_euler_gen}), (\ref{eq:sum_phi_n_k}), and (\ref{eq:sum_w}) we obtain
\begin{flalign}
\label{eq:3.one_first_line}
\sum_{b_1=1}^{\m-1}\sum_{\substack{b_2=\left\lfloor\frac{\n b_1}{\m}+1\right\rfloor \\ b_2 \perp b_1}}^{\n}1 &=\sum_{b_1=1}^{\m-1}\left(\n\frac{ \phi (b_1)}{b_1} - \frac{\n }{\m}\phi (b_1) + O\left(2^{w(b_1)}\right)\right)\\
\nonumber&= \n\left(\frac{6}{\pi^2}\m + O(\log\m)\right) - \frac{\n}{\m}\left(\frac{3}{\pi^2}\m^2 + O(\m\log\m)\right) + O(\m\log\m)\\
\label{eq:3.one_ver1}
&= \frac{3}{\pi^2}\m\n+O(\n\log \m) + O(\m \log \m).
\end{flalign}
Changing the order of summation in the above sum, we deduce the same result, but with a slightly different error term:
\begin{flalign}
\label{eq:3.one_ver2}
\nonumber \sum_{b_1=1}^{\m-1}\sum_{\substack{b_2=\left\lfloor\frac{\n b_1}{\m}+1\right\rfloor \\ b_2 \perp b_1}}^{\n}1 &=\sum_{b_2=1}^{\n}\sum_{\substack{b_1=1 \\ b_1 \perp b_2}}^{\left\lceil\frac{\m b_2}{\n}-1\right\rceil} 1\\
\nonumber&=\sum_{b_2=1}^{\n}\left(\frac{ \phi (b_2)}{b_2}\left\lceil\frac{\m b_2 }{\n}-1\right\rceil + O\left(2^{w(b_2)}\right)\right)\\
\nonumber&=\sum_{b_2=1}^{\n}\left(\frac{\m }{\n}\phi(b_2) + O\left(\frac{\phi (b_2)}{b_2}\right) + O\left(2^{w(b_2)}\right)\right)\\
\nonumber&=\frac{\m }{\n}\left(\frac{3}{\pi^2}\n^2 + O\left(\n \log \n\right)\right) + O(\n) + O\left(\n \log \n\right)\\
&= \frac{3}{\pi^2}\m\n+O(\m\log \n) + O(\n \log \n).
\end{flalign}
Finally, denoting $\alpha = \max(\m,\n)$ and $\beta = \min(\m,\n)$ we derive from (\ref{eq:3.one_ver1}) and (\ref{eq:3.one_ver2}) 
\begin{flalign}
\nonumber\sum_{b_1=1}^{\m-1}\sum_{\substack{b_2=\left\lfloor\frac{\n b_1}{\m}+1\right\rfloor \\ b_2 \perp b_1}}^{\n} 1 &= \frac{3}{\pi^2}\m\n + O(\alpha\log \beta) + O(\beta\log \beta) \\
\label{eq:3.one_joined}
&= \frac{3}{\pi^2}\m\n + O(\m \log \n + \n \log \m),
\end{flalign}
and hence
\begin{flalign}
\label{eq:3.first}
\sum_{b_1=1}^{\m-1}\sum_{\substack{b_2=\left\lfloor\frac{\n b_1}{\m}+1\right\rfloor \\ b_2 \perp b_1}}^{\n} 2\m\n =  \frac{6}{\pi^2}\m^2\n^2 + O(\m^2 \n \log \n + \m \n^2 \log \m).
\end{flalign}

\noindent
\textbf{2. Estimation of $\sum\sum \n b_1$.}

Using formulas (\ref{eq:3.one_first_line}) and (\ref{eq:sum_phi_n_k}) we obtain
\begin{flalign}
\label{eq:sum_b_1_ver1}
\nonumber  \sum_{b_1=1}^{\m-1}\sum_{\substack{b_2=\left\lfloor\frac{\n b_1}{\m}+1\right\rfloor \\ b_2 \perp b_1}}^{\n}b_1 &= \sum_{b_1=1}^{\m-1}b_1\sum_{\substack{b_2=\left\lfloor\frac{\n b_1}{\m}+1\right\rfloor \\ b_2 \perp b_1}}^{\n}1 \\
\nonumber&=\sum_{b_1=1}^{\m-1}\left(\n\phi (b_1) - \frac{\n }{\m}b_1\phi (b_1) + O\left(b_12^{w(b_1)}\right)\right)\\
\nonumber&=\frac{3}{\pi^2}\m^2\n + O(\m\n\log\m) - \frac{\n}{\m}\left( \frac{2}{\pi^2}\m^3 + O(\m^2\log \m)\right)+ O(\m^2 \log \m)\\
&= \frac{1}{\pi^2}\m^2\n+O(\m \n\log \m) + O(\m^2  \log \m).
\end{flalign}

Again, changing the order of summation in the above sum, we deduce the same result with a different error term:
\begin{flalign}
\label{eq:sum_b_1_ver2}
\nonumber  \sum_{b_1=1}^{\m-1}\sum_{\substack{b_2=\left\lfloor\frac{\n b_1}{\m}+1\right\rfloor \\ b_2 \perp b_1}}^{\n}b_1 &= \sum_{b_2=1}^{\n}\sum_{\substack{b_1=1 \\ b_1 \perp b_2}}^{\left\lceil\frac{\m b_2}{\n}-1\right\rceil} b_1 \\
\nonumber&= \sum_{b_2=1}^{\n}\left(\frac{\phi(b_2)}{2b_2}\left\lceil\frac{\m b_2}{\n} -1 \right\rceil^2 + \frac{\m b_2}{\n}O\left(2^{w(b_2)}\right) + O\left(2^{w(b_2)}\right)\right)\\
\nonumber&= \sum_{b_2=1}^{\n}\left( \frac{\m^2}{2\n^2}b_2\phi(b_2) +\frac{\m}{2\n}\phi(b_2) + \frac{\m b_2}{\n}O\left(2^{w(b_2)}\right) + O\left(2^{w(b_2)}\right)\right)\\
&= \frac{1}{\pi^2}\m^2\n + O(\m^2  \log \n) +O(\m \n\log \n).
\end{flalign}

Finally, denoting $\alpha = \max(\m,\n)$ and $\beta = \min(\m,\n)$ we derive from (\ref{eq:sum_b_1_ver1}) and (\ref{eq:sum_b_1_ver2}) 
\begin{flalign}
\nonumber\sum_{b_1=1}^{\m-1}\sum_{\substack{b_2=\left\lfloor\frac{\n b_1}{\m}+1\right\rfloor \\ b_2 \perp b_1}}^{\n}b_1 = \frac{1}{\pi^2}\m^2\n + \m\left(O(\alpha \log \beta) + O(\beta \log \beta)\right) = \frac{1}{\pi^2}\m^2\n + o(\m^2\n),
\end{flalign}
and hence

\begin{flalign}
\label{eq:3.b_1}
\sum_{b_1=1}^{\m-1}\sum_{\substack{b_2=\left\lfloor\frac{\n b_1}{\m}+1\right\rfloor \\ b_2 \perp b_1}}^{\n}\n b_1 = \frac{1}{\pi^2}\m^2\n^2 + o(\m^2\n^2)
\end{flalign}

\noindent
\textbf{3. Estimation of $\sum\sum \m b_2$.}

Using formulas  (\ref{eq:sum_euler_gen}), (\ref{eq:sum_w}), and (\ref{eq:sum_phi_n_k}) we obtain:
\begin{flalign}
\label{eq:3.2}
\nonumber\sum_{b_1=1}^{\m-1}\sum_{\substack{b_2=\left\lfloor\frac{\n b_1}{\m}+1\right\rfloor \\ b_2 \perp b_1}}^{\n} b_2 &= \sum_{b_1=1}^{\m-1}\left(\frac{\phi(b_1)\n^2}{2b_1} - \frac{\phi(b_1)b_1\n^2}{2\m^2} + O\left(\n 2^{w(b_1)}\right)\right)\\
\nonumber&=\frac{\n^2}{2}\sum_{b_1=1}^{\m-1}\left(\frac{\phi(b_1)}{b_1} - \frac{1}{\m^2}b_1\phi(b_1)\right)+O(\m\n\log \m)\\
\nonumber&= \frac{\n^2}{2}\left(\frac{6}{\pi^2}\m + O(\log\m)- \frac{1}{\m^2}\left(\frac{2}{\pi^2}\m^3 + O(\m^2\log \m)\right)\right) + O(\m\n\log \m)\\
&= \frac{2}{\pi^2}\m\n^2 + O(\m\n\log\m) + O(\n^2\log\m).
\end{flalign}
Similarly to the previous case, by changing the order of summation, one can show that 
$$
	\sum_{b_1=1}^{\m-1}\sum_{\substack{b_2=\left\lfloor\frac{\n b_1}{\m}+1\right\rfloor \\ b_2 \perp b_1}}^{\n} b_2 = 
	\frac{2}{\pi^2}\m\n^2 + O(\m\n\log\n) + O(\n^2\log\n),
$$
which together with (\ref{eq:3.2}) imply
\begin{flalign}
\nonumber\sum_{b_1=1}^{\m-1}\sum_{\substack{b_2=\left\lfloor\frac{\n b_1}{\m}+1\right\rfloor \\ b_2 \perp b_1}}^{\n} b_2 = \frac{2}{\pi^2}\m\n^2 + o(\m\n^2),
\end{flalign}

and hence
\begin{flalign}
\label{eq:m_b2}
\sum_{b_1=1}^{\m-1}\sum_{\substack{b_2=\left\lfloor\frac{\n b_1}{\m}+1\right\rfloor \\ b_2 \perp b_1}}^{\n} \m b_2 = \frac{2}{\pi^2}\m^2\n^2 + o(\m^2\n^2).
\end{flalign}

\noindent
\begin{minipage}{\textwidth}
\noindent
\textbf{4. Estimation of $\sum\sum \frac{\n^2 b_1}{b_2}$.}

Using formulas (\ref{eq:sum_euler_divide_new}), (\ref{eq:sum_phi_ln}), (\ref{eq:sum_phi_n_k}), and the fact that 
$\log \left\lfloor x \right\rfloor = \log x + O\left(\frac{1}{x}\right)$ for $x \geq 1$, we obtain
\begin{flalign}
\nonumber&\sum_{b_1=1}^{\m-1}\sum_{\substack{b_2=\left\lfloor\frac{\n b_1}{\m}+1\right\rfloor \\ b_2 \perp b_1}}^{\n}\frac{b_1}{b_2} =\sum_{b_1=1}^{\m-1}b_1 \left(\sum_{\substack{b_2=1 \\ b_2 \perp b_1}}^{\n}\frac{1}{b_2} - \sum_{\substack{b_2=1 \\ b_2 \perp b_1}}^{\left\lfloor\frac{\n b_1}{\m}\right\rfloor}\frac{1}{b_2}\right)\\
\nonumber&= \sum_{b_1=1}^{\m-1}b_1\left(\frac{\phi(b_1)}{b_1}\log \n +\frac{\phi(b_1)}{b_1} O\left(\frac{1}{\n}\right) - \frac{\phi(b_1)}{b_1}\log \left\lfloor\frac{\n b_1}{\m}\right\rfloor - \frac{\phi(b_1)}{b_1}O\left(\frac{\m}{\n b_1}\right)\right)\\
\nonumber&= \sum_{b_1=1}^{\m-1}\phi(b_1)\left(\log \n - \log \frac{\n b_1}{\m} + O\left(\frac{1}{\n}\right) + O\left(\frac{\m}{\n b_1}\right)\right)\\
\nonumber&= \sum_{b_1=1}^{\m-1}\phi(b_1)\left(\log \m - \log b_1 + O\left(\frac{1}{\n}\right) + O\left(\frac{\m}{\n b_1}\right)\right)\\
\nonumber&= \frac{3}{\pi^2}\m^2\log \m + O(\m \log^2\m) - \frac{3}{\pi^2}\m^2\log \m + \frac{3}{2\pi^2}\m^2 + o(\m^2) + O\left(\frac{\m^2}{\n}\right)\\
\label{eq:b1/b2}
& = \frac{3}{2\pi^2}\m^2 + o(\m^2) + O\left(\frac{\m^2}{\n}\right),
\end{flalign}
and hence
\begin{flalign}
\label{eq:3.3}
\sum_{b_1=1}^{\m-1}\sum_{\substack{b_2=\left\lfloor\frac{\n b_1}{\m}+1\right\rfloor \\ b_2 \perp b_1}}^{\n}\frac{\n^2 b_1}{b_2} = \frac{3}{2\pi^2}\m^2\n^2 + o(\m^2\n^2).
\end{flalign}

\noindent
Finally, combining (\ref{eq:3.0}), (\ref{eq:3.first}), (\ref{eq:3.b_1}), (\ref{eq:m_b2}), and (\ref{eq:3.3}) we derive
\begin{flalign}
\nonumber|Z_2^c(\m,\n)| &= 
\frac{12}{\pi^2}\left(\frac{6}{\pi^2}\m^2\n^2 +\frac{1}{\pi^2}\m^2\n^2
-\frac{2}{\pi^2}\m^2\n^2 
-\frac{3}{2\pi^2}\m^2\n^2
 +  o(\m^2\n^2)\right) \\
\nonumber&= \frac{42}{\pi^4}\m^2\n^2 + o(\m^2\n^2).
\end{flalign}
\end{minipage}

\end{proof}

\noindent
\begin{minipage}{\textwidth}
\subsection{The number of pairs of segments with one corner point}
\label{sec:one_corner_point}

There is no loss of generality in assuming $\{A,B,C,D\} \cap \Vertic(\rMN) = \{A\}$ for every pair $\{AB, CD\} \in Z_{1}(\m,\n)$.
We consider the partition of $Z_1(\m,\n)$ into the following two subsets:

\begin{enumerate}
\item  $Z_1^a(\m,\n)$ the set of those pairs $\{AB,CD\}$ in which the point $B$ is an interior point of $\rMN$;
\item $Z_1^b(\m,\n)$ the set of those pairs $\{AB,CD\}$ in which the point $B$ belongs to the boundary of $\rMN$.
\end{enumerate}
\end{minipage}

\noindent
In the rest of the section we estimate the sizes of these sets in separate lemmas.

\begin{lemma}\label{lem:Z1a}
$
|Z_1^a(\m,\n)| = \frac{72}{\pi^4}\m^2\n^2 + o(\m^2\n^2).
$
\end{lemma}
\begin{proof}
Due to symmetry, for a corner point $R$ of $\rMN$ the number of pairs $\{AB, CD\} \in Z_1^a(\m,\n)$, where $A$ coincides with 
$R$, is the same for every \linebreak $R \in \{ R_1, R_2, R_3, R_4 \}$. 
Therefore, it is enough to estimate the number of pairs where $A$ coincides with a fixed corner point of $\rMN$,
and we assume that $A = R_1$.

Since $B$ is an interior point of $\rMN$ and neither $C$ nor $D$ is a corner point of $\rMN$, we conclude
that one of $C$ and $D$ belongs to the interior of $R_2R_3$ and the other belongs to the interior of $R_3R_4$.
Without loss of generality, we assume that $C$ is an interior point of $R_3R_4$ and $D$ is an interior point of $R_2R_3$.

Under the above assumptions, we will first estimate the number of pairs in $Z_1^a(\m,\n)$ in which $B = (b_1,b_2)$
belongs to the triangle $AR_2R_3$. Notice that the latter assumption is equivalent to the inequality $\frac{b_1}{b_2} \leq \frac{\m}{\n}$.
\begin{figure}
\centering
\begin{tikzpicture}[scale=1.2]

	\draw [<->,thick] (0,4) node (yaxis) [above] {}
       |- (6,0) node (xaxis) [right] {};
				
	\draw (0,0) -- (5,0);
	\draw (0,0) -- (0,3);
	\draw (0,3) -- (5,3);
	\draw (5,0) -- (5,3);
	\draw (0,0) -- (3,3);
	\draw (0,0) -- (5,3);
	\draw[black,fill=white] (0,0) circle (1.4pt);
	\draw (0,0) node [left] {$A=R_1$};
	\draw [line width=0.5mm](3,3) -- (5,3);
	\draw[black,fill=white] (3,3) circle (1.4pt);
	\draw (3,3) node [above] {$X$};
	\draw[black,fill=white] (4,3) circle (1.4pt);
	\draw (4,3) node [above] {$D$};
	\draw[black,fill=white] (2,2) circle (1.4pt);
	\draw (2,2) node [above] {$B$};
	\draw[black,fill=white] (0,3) circle (1.4pt);
	\draw (0,3) node [left] {$R_2$};
	\draw[black,fill=white] (5,0) circle (1.4pt);
	\draw (5,0) node [below] {$R_4$};
	\draw [line width=0.5mm](5,0) -- (5,3);
	\draw[black,fill=white] (5,1) circle (1.4pt);
	\draw (5,1) node [right] {$C$};
	\draw[black,fill=white] (5,0) circle (1.4pt);
	\draw (5,0) node [below] {$R_4$};
	\draw[black,fill=white] (5,3) circle (1.4pt);
	\draw (5,3) node [right] {$R_3$};
	
\end{tikzpicture}
\caption{The case $B \in \protect\tr{A}{R_2}{R_3}$. 
The points $C$ and $D$ belong to the segments $R_3R_4$ and $XR_3$ respectively, where $X = \protect\ell(AB) \cap R_2R_3$.}
\label{fig:main_5}
\end{figure}
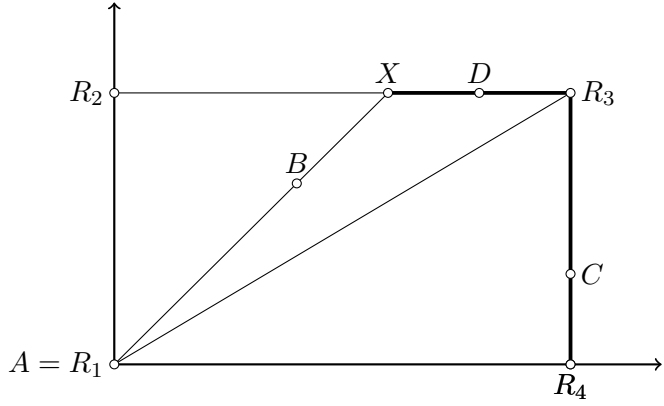
Let us denote 
$$
X =\ell(AB) \cap R_2R_3= \left(\frac{\n b_1}{b_2},\n\right).
$$
It follows from \cref{cor:P+-} that $AB$ and $CD$ are in convex position if and only if $D$ is an interior point of $XR_3$ (see \cref{fig:main_5}).
Therefore, by denoting $D = (d_1, \n)$ and $C = (\m, c_2)$, the number of desired prime pairs segments can be expressed as
\begin{flalign}
\label{eq:5.main}
&\sum_{b_1=1}^{\m-1}\sum_{\substack{b_2=\left\lfloor\frac{\n b_1}{\m} + 1\right\rfloor \\ b_2 \perp b_1}}^{\n-1}\sum_{c_2=1}^{\n-1}\sum_{\substack{d_1=\left\lfloor\frac{\n b_1}{b_2} + 1\right\rfloor \\ (\m - d_1) \perp (\n-c_2) }}^{\m-1}1.
\end{flalign}
We start by estimating the contribution of the latter two sums.
\begin{flalign}
\label{eq:5.1}
\sum_{c_2=1}^{\n-1}\sum_{\substack{d_1=\left\lfloor\frac{\n b_1}{b_2} + 1\right\rfloor \\ (\m - d_1) \perp (\n-c_2) }}^{\m-1}1 &= \sum_{c_2'=1}^{\n}\sum_{\substack{d_1'=1  \\ d_1' \perp (c_2') }}^{\m-\left\lfloor\frac{\n b_1}{b_2}+1\right\rfloor}1 + O(\m)\\
&\nonumber\overset{\mathrm{(\ref{eq:sum_euler_gen})}}{=}
\sum_{c_2'=1}^{\n}\left(\frac{\phi(c_2')}{c_2'}\left(\m - \n\frac{b_1}{b_2} + O(1)\right) + O\left(2^{w\left(c_2'\right)}\right)\right) + O(\m)\\
\nonumber&\overset{\mathrm{(\ref{eq:sum_phi_n_k})}}{=}\left(\frac{6\n}{\pi^2} + O(\log \n)\right)\left(\m - \n\frac{b_1}{b_2} + O(1) \right) +O(\n \log \n) + O(\m)\\
\nonumber&= \frac{6\n}{\pi^2}\left(\m - \n\frac{b_1}{b_2}\right) +O(\m \log \n) + O\left(\frac{b_1}{b_2}\n \log \n\right) + O(\n \log \n)\\
\nonumber&= \frac{6\n}{\pi^2}\left(\m - \n\frac{b_1}{b_2}\right) + O(\n \log \n) +O(\m \log \n).
\end{flalign}
By changing the order of summation in (\ref{eq:5.1}), one can show that 
\begin{flalign}
\nonumber\sum_{c_2=1}^{\n-1}\sum_{\substack{d_1=\left\lfloor\frac{\n b_1}{b_2} + 1\right\rfloor \\ (\m - d_1) \perp (\n-c_2) }}^{\m-1}1 = \frac{6\n}{\pi^2}\left(\m - \n\frac{b_1}{b_2}\right) + O(\m \log \n + \n \log \m).
\end{flalign}
Now, plugging in the above result to (\ref{eq:5.main}) and using formulas (\ref{eq:3.one_joined}) and (\ref{eq:b1/b2}) we obtain:
\begin{flalign*}
&\sum_{b_1=1}^{\m-1}\sum_{\substack{b_2=\left\lfloor\frac{\n b_1}{\m} + 1\right\rfloor \\ b_2 \perp b_1}}^{\n-1}\sum_{c_2=1}^{\n-1}\sum_{\substack{d_1=\left\lfloor\frac{\n b_1}{b_2} + 1\right\rfloor \\ (\m - d_1) \perp (\n-c_2) }}^{\m-1}1\\
&= \sum_{b_1=1}^{\m-1}\sum_{\substack{b_2=\left\lfloor\frac{\n b_1}{\m} + 1\right\rfloor \\ b_2 \perp b_1}}^{\n-1}\left(\frac{6\n}{\pi^2}\left(\m - \n\frac{b_1}{b_2}\right) + O(\m \log \n + \n \log \m)\right)\\
\nonumber&= \frac{6\n}{\pi^2}\left(\m\left(\frac{3}{\pi^2}\m\n + O(\m \log \n + \n \log \m)\right)- \n\left(\frac{3}{2\pi^2}\m^2 + o(\m^2) + O\left(\frac{\m^2}{\n}\right)\right)\right)\\
&+ O(\m^2\n \log \n + \m \n^2 \log \m)\\
\nonumber&=\frac{9}{\pi^4}\m^2\n^2 + \n^2o(\m^2) + O(\m^2\n \log \n + \m \n^2 \log \m).
\end{flalign*}

Note that the obtained estimation is symmetric with respect to $\m$ and $\n$, which implies that 
the number of pairs in $Z_1^a(\m,\n)$ in which $B$ belongs to $\tr{A}{R_3}{R_4}$ has the same asymptotics.
Therefore, taking into account additionally all symmetric cases corresponding to the location of $A$, we finally conclude that 
\begin{equation}
	\nonumber |Z_1^a(\m,\n)| = \frac{72}{\pi^4}\m^2\n^2 + o(\m^2\n^2).
\end{equation}

\end{proof}

\begin{lemma}\label{lem:Z1b}
$$
|Z_1^b(\m,\n)| = \frac{6\n^2}{\pi^2}\sum_{\substack{b_1=1 \\ b_1 \perp \n}}^{\m}\left(2\m - b_1\right) + \frac{6\m^2}{\pi^2}\sum_{\substack{c_2=1 \\ c_2 \perp \m}}^{\n}\left(2\n - c_2\right) + O(\m^2\n+\m\n^2).
$$
\end{lemma}
\begin{proof}
As in the proof of \cref{lem:Z1a}, 
due to symmetry, for a corner point $R$ of $\rMN$ the number of pairs $\{AB, CD\} \in Z_1^b(\m,\n)$, where $A$ coincides with 
$R$, is the same for every $R \in \{ R_1, R_2, R_3, R_4 \}$. Therefore, it is enough to estimate the number of pairs where $A$ 
coincides with a fixed corner point of $\rMN$, and we assume that $A = R_1$.

It is easy to see that if $B$ is an internal point of $R_1R_2$ or $R_1R_4$, then one of $C$ and $D$ belongs to the interior
of $R_2R_3$ and the other belongs to the interior of $R_3R_4$. Therefore, taking into account primality of $AB$,
the number of pairs, in which $B$ is an internal point of $R_1R_2$ or $R_1R_4$, is $O(\m\n)$. The latter does not 
affect the asymptotics, and without loss of generality we assume from now on that $B$ is an internal point of one of 
the sides $R_2R_3$ and $R_3R_4$.

Suppose first that $B$ is an internal point of $R_2R_3$, i.e.  $B = (b_1,\n)$ for some $0 < b_1 < \m$, and $b_1 \perp \n$.
Then $AB \cap R_3R_4 = \emptyset$, and hence $CD \cap R_3R_4 \neq \emptyset$, which implies that
either $C$ or $D$ belongs to the interior of $R_3R_4$.
\begin{figure}
\centering
\begin{tikzpicture}[scale=1.2]

	\draw [<->,thick] (0,4) node (yaxis) [above] {}
       |- (6,0) node (xaxis) [right] {};
				
	\draw (0,0) -- (5,0);
	\draw (0,0) -- (0,3);
	\draw (0,3) -- (5,3);
	\draw (5,0) -- (5,3);
	\draw (0,0) -- (3,3);
	\draw[black,fill=light_grey] (0,0) -- (5,1) -- (5,0);
	\draw[black,fill=light_grey] (3,3) -- (5,3) -- (5,1);
	\draw (5,1) -- (3,3);
	\draw [line width=0.5mm](0,3) -- (5,3);
	\draw [line width=0.5mm](5,0) -- (5,3);
	\draw[black,fill=white] (0,0) circle (1.4pt);
	\draw (0,0) node [left] {$A=R_1$};
	\draw[black,fill=white] (3,3) circle (1.4pt);
	\draw (3,3) node [above] {$B$};
	\draw[black,fill=white] (0,3) circle (1.4pt);
	\draw (0,3) node [left] {$R_2$};
	\draw[black,fill=white] (5,3) circle (1.4pt);
	\draw (5,3) node [right] {$R_3$};
	\draw[black,fill=white] (5,0) circle (1.4pt);
	\draw (5,0) node [below] {$R_4$};
	\draw[black,fill=white] (5,1) circle (1.4pt);
	\draw (5,1) node [right] {$C$};
	\draw[black,fill=white] (5,0) circle (1.4pt);
	\draw (5,0) node [below] {$R_4$};
	
\end{tikzpicture}
\caption{The points $B$ and $C$ belong to the segments $R_2R_3$ and $R_3R_4$ respectively, the grey triangles form the admissible area for $D$.}
\label{fig:main_2}
\end{figure}
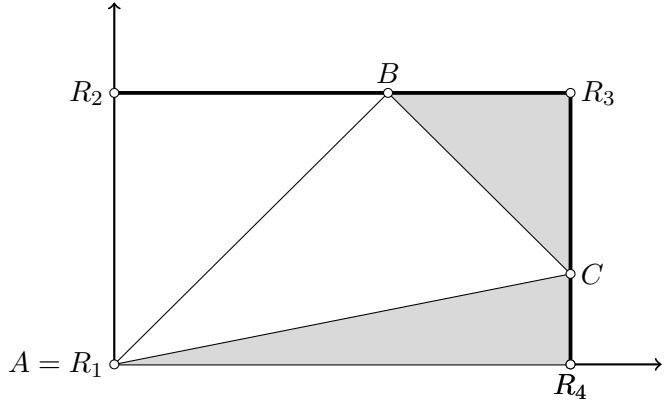
Without loss of generality we assume the former, i.e. $C = (\m, c_2)$ for some $0 < c_2 < \n$ (see \cref{fig:main_2}).
Under these assumptions, \cref{cor:P+-} implies that $AB$ and $CD$ are in convex position if and only if the point $D$ belongs to 
$\tr{B}{C}{R_3} \setminus (BC \cup \{R_3\})$ or $\tr{A}{C}{R_4} \setminus (AC \cup \{R_4\})$.
Therefore, using \cref{lem:S_triangle} and \cref{cor:S-triangle}, we conclude that the number of such pairs of prime segments is
\begin{flalign*}
\nonumber&\sum_{\substack{b_1=1 \\ b_1 \perp \n}}^{\m-1}\sum_{c_2=1}^{\n-1}\frac{6}{\pi^2}\left(\Area(\tr{B}{C}{R_3}) + \Area(\tr{A}{C}{R_4}) + O(\m+\n)\right)\\
\nonumber&= \frac{3}{\pi^2}\sum_{\substack{b_1=1 \\ b_1 \perp \n}}^{\m}\sum_{c_2=1}^{\n}((\m-b_1)(\n-c_2) + \m c_2) + O(\m^2\n + \m\n^2)\\
\nonumber&= \frac{3}{\pi^2}\sum_{\substack{b_1=1 \\ b_1 \perp \n}}^{\m}\sum_{c_2=1}^{\n}(\m\n - \n b_1 + b_1c_2 )  + O(\m^2\n + \m\n^2)\\
\nonumber&= \frac{3}{\pi^2}\sum_{\substack{b_1=1 \\ b_1 \perp \n}}^{\m}\left((\m\n - \n b_1)\n + b_1\left(\frac{\n^2}{2} + O(\n)\right)\right) + O(\m^2\n + \m\n^2)\\
&= \frac{3\n^2}{\pi^2}\sum_{\substack{b_1=1 \\ b_1 \perp \n}}^{\m}\left(\m - \frac{b_1}{2}\right) + O(\m^2\n + \m\n^2).
\end{flalign*}

By symmetry, the number of pairs in which $B$ is an internal point of $R_3R_4$ is
$$
\frac{3\m^2}{\pi^2}\sum_{\substack{c_2=1 \\ c_2 \perp \m}}^{\n}\left(\n - \frac{c_2}{2}\right) + O(\m^2\n + \m\n^2).
$$

Putting all together and taking into account the symmetric cases corresponding to the location of $A$, we finally conclude that 
\begin{equation}
\nonumber |Z_1^b(\m,\n)| = \frac{6\n^2}{\pi^2}\sum_{\substack{b_1=1 \\ b_1 \perp \n}}^{\m}\left(2\m - b_1\right) + \frac{6\m^2}{\pi^2}\sum_{\substack{c_2=1 \\ c_2 \perp \m}}^{\n}\left(2\n - c_2\right) + O(\m^2\n + \m\n^2).
\end{equation}
\end{proof}

We note that in \cref{lem:Z1b} we deliberately did not compute a closed-form asymptotic, as the
obtained formula will be crucial later to obtain a better error term.

\subsection{The number of pairs of segments with no corner points}
\label{sec:0points}

In this section we estimate the size of $Z_0(\m,\n)$, i.e. the number of those pairs of segments in $Z(\m,\n)$
none of whose endpoints is a corner of $\rMN$.

\begin{lemma}\label{lem:Z0}
$
|Z_0(\m,\n)| = \frac{72}{\pi^4}\m^2\n^2 + O(\m^2 \n \log \n).
$
\end{lemma}
\begin{proof}
Let $\{AB, CD\}$ be an arbitrary pair in $Z_0(\m, \n)$. 
The fact that none of the points $A, B, C,$ and $D$ is a corner of $\rMN$ implies
that each of the sides of $\rMN$ contains exactly one of these points. 
Furthermore, since $AB$ and $CD$ are in convex position, 
we conclude that the endpoints of the same segment belong to the adjacent sides of $\rMN$.
Therefore, without loss of generality we can assume $A \in R_1R_2$ and $C \in R_3R_4$, in which case
either $B \in R_2R_3$ and $D \in R_1R_4$, or $B \in R_1R_4$ and $D \in R_2R_3$.
The two cases are symmetric and we assume the former one, i.e. $B \in R_2R_3, D \in R_1R_4$ (see \cref{fig:main_1}).
Let us denote $A = (0, a_2), C = (\m, c_2), B = (b_1, \n),$ and $D = (d_1, 0)$.

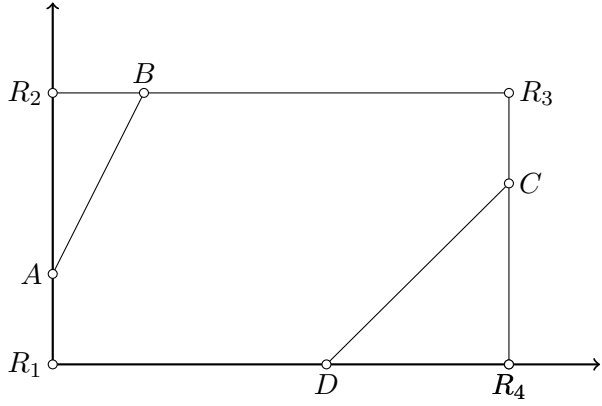
\begin{figure}
\centering
\begin{tikzpicture}[scale=1.2]

	\draw [<->,thick] (0,4) node (yaxis) [above] {}
       |- (6,0) node (xaxis) [right] {};
				
	\draw (0,0) -- (5,0);
	\draw (0,0) -- (0,3);
	\draw (0,3) -- (5,3);
	\draw (5,0) -- (5,3);
	\draw (0,1) -- (1,3);
	\draw (5,2) -- (3,0);
	\draw[black,fill=white] (0,0) circle (1.4pt);
	\draw (0,0) node [left] {$R_1$};
	\draw[black,fill=white] (0,1) circle (1.4pt);
	\draw (0,1) node [left] {$A$};
	\draw[black,fill=white] (1,3) circle (1.4pt);
	\draw (1,3) node [above] {$B$};
	\draw[black,fill=white] (5,2) circle (1.4pt);
	\draw (5,2) node [right] {$C$};
	\draw[black,fill=white] (3,0) circle (1.4pt);
	\draw (3,0) node [below] {$D$};
	\draw[black,fill=white] (0,3) circle (1.4pt);
	\draw (0,3) node [left] {$R_2$};
	\draw[black,fill=white] (5,3) circle (1.4pt);
	\draw (5,3) node [right] {$R_3$};
	\draw[black,fill=white] (5,0) circle (1.4pt);
	\draw (5,0) node [below] {$R_4$};
	\draw[black,fill=white] (5,0) circle (1.4pt);
	\draw (5,0) node [below] {$R_4$};
	
\end{tikzpicture}
\caption{Each of $A$, $B$, $C$, and $D$ belongs to a unique side of $\rMN$. 
The endpoints of the same segment belong to the adjacent sides of $\rMN$.}
\label{fig:main_1}
\end{figure}
Under the above assumptions the segments $AB$ and $CD$ are prime if and only if $(\n-a_2) \perp b_1$ and $(\m-d_1) \perp c_2$, and the number of such pairs is
\begin{flalign}
\nonumber\sum_{a_2 = 1}^{\n-1}\sum_{\substack{b_1 = 1 \\ b_1 \perp (\n-a_2)}}^{\m-1}\sum_{c_2 = 1}^{\n-1}\sum_{\substack{d_1 = 1 \\ (\m-d_1) \perp c_2}}^{\m-1}1 = \sum_{a'_2 = 1}^{\n-1}\sum_{\substack{b_1 = 1 \\ b_1 \perp a'_2}}^{\m-1}\sum_{c_2 = 1}^{\n-1}\sum_{\substack{d'_1 = 1 \\ d'_1 \perp c_2}}^{\m-1}1.
\end{flalign}

Using formula (\ref{eq:sum_coprime}) we obtain
\begin{flalign}
\nonumber\sum_{a'_2 = 1}^{\n-1}\sum_{\substack{b_1 = 1 \\ b_1 \perp a'_2}}^{\m-1}\sum_{c_2 = 1}^{\n-1}\sum_{\substack{d'_1 = 1 \\ d'_1 \perp c_2}}^{\m-1}1 &= \sum_{a'_2 = 1}^{\n-1}\sum_{\substack{b_1 = 1 \\ b_1 \perp a'_2}}^{\m-1}\left(\frac{6}{\pi^2}\m\n + O(\m \log \n)\right)
=\frac{36}{\pi^4}\m^2\n^2 + O(\m^2 \n \log \n).
\end{flalign}

Finally, taking into account the symmetric case of $B \in R_1R_4$ and $D \in R_2R_3$, we derive the desired result
\begin{equation}
\nonumber|Z_0(\m,\n)| = \frac{72}{\pi^4}\m^2\n^2 + O(\m^2 \n \log \n).
\end{equation}
\end{proof}

\subsection{Summarizing results}

In the following theorem we prove the main result of the paper by putting everything together.

\begin{theorem}
\label{th:s_general_formula}
$$
p(m,n) = \frac{25}{12\pi^4}m^4n^4 + o(m^4n^4).
$$
\end{theorem}
\begin{proof}
First, using (\ref{eq:sumZ}) and Lemmas \ref{lem:Z34} and \ref{lem:Z2a}, we expand formula (\ref{eq:total_1}) as follows:
\begin{flalign*}
\nonumber & p(m,n) &= \sum_{\m=1}^{m-1}(m-\m)\sum_{\n=1}^{n-1}(n-\n) &\cdot |Z(\m,\n)|&\\
\nonumber & &=\sum_{\m=1}^{m-1}(m-\m)\sum_{\n=1}^{n-1}(n-\n)&\cdot\Big(|Z_4(\m,\n)| + |Z_3(\m,\n)| + |Z_2(\m,\n)| &\\
\nonumber & & & + |Z_1(\m,\n)| + |Z_0(\m,\n)|\Big) &\\
\nonumber & &=\sum_{\m=1}^{m-1}(m-\m)\sum_{\n=1}^{n-1}(n-\n)&\cdot\Big(|Z_2^b(\m,\n)| + |Z_2^c(\m,\n)| + |Z_1^a(\m,\n)| &\\
\nonumber & & &+ |Z_1^b(\m,\n)| + |Z_0(\m,\n)|\Big) + o(m^4n^4).&
\end{flalign*}

Next,  different parts of the above sum can be estimated separately, as 
follows\footnote{Due to their simplicity and technicality, the proofs are excluded from the main part of the paper and included in the appendix.}.

\medskip
\noindent 
\begin{flalign}
\label{eq:final_1_1}
\sum_{\m=1}^{m-1}(m-\m)\sum_{\substack{\n=1}}^{n-1}(n-\n) \cdot |Z_2^b(\m,\n)| = \frac{1}{24\pi^4}m^4n^4 + o(m^4n^4).
\end{flalign}

\begin{flalign}
\label{eq:final_Z2_1_0_1}
\sum_{\m=1}^{m-1}(m-\m)\sum_{\n=1}^{n-1}(n-\n)\Big( |Z_2^c(\m,\n)| + |Z_1^a(\m,\n)| + |Z_0(\m,\n)| \Big) \ = \ \frac{31}{24\pi^4}m^4n^4 + o(m^4n^4).
\end{flalign}

\begin{flalign}
\label{eq:big_sum_estimation_1}
\sum_{\m=1}^{m-1}(m-\m)\sum_{\n=1}^{n-1}(n-\n) \cdot |Z_1^b(\m,\n)| = \frac{3}{4\pi^4}m^4n^4+ o(m^4n^4).
\end{flalign}

Finally, plugging in (\ref{eq:final_1_1}), (\ref{eq:final_Z2_1_0_1}), and (\ref{eq:big_sum_estimation_1}) into the initial formula we obtain
$$
p(m,n) = \frac{1}{24\pi^4}m^4n^4 + \frac{31}{24\pi^4}m^4n^4 +  \frac{3}{4\pi^4}m^4n^4 + o(m^4n^4)= \frac{25}{12\pi^4}m^4n^4 + o(m^4n^4).
$$
\end{proof}

The following corollary from \cref{th:t_2=p}, \cref{cor:t_2-p}, and \cref{th:s_general_formula} reveal the relations between $2$-threshold functions, proper pairs of segments, and pairs of prime segments in convex position.

\begin{corollary}\label{cor:bijection_1}
	The number of 2-threshold functions is asymptotically equivalent to 
	the number of proper pairs of segments and to the number of pairs of prime segments in convex position, i.e., 
	$$
		t_2(m,n) = (1 + o(1)) q(m,n)
	$$
	and 
	$$
		t_2(m,n) = (1 + o(1)) p(m,n).
	$$
\end{corollary}

The above corollary implies \cref{th:t_k}.

\section{On the number of $k$-threshold functions for $k \geq 3$}
\label{sec:k-threshold-number}

The obtained asymptotic formula for the number of $2$-threshold functions can be used to improve upper bound (\ref{eq:k_trivial_upper_bound}) on the number of $k$-threshold functions for $k \geq 3$.
Indeed, since a $k$-threshold function can be seen as a conjunction of several $2$-threshold functions and at most one threshold function, we have:
\begin{flalign}
\nonumber t_k(m,n) \leq \displaystyle{t_2(m,n) \choose \frac{k}{2}} & = \frac{t_2(m,n)^{\frac{k}{2}}}{\frac{k}{2}!} + o\left(m^{2k}n^{2k}\right) \\
\label{eq:odd_k}
& = \frac{5^k}{12^{\frac{k}{2}}\pi^{2k}\frac{k}{2}!}m^{2k}n^{2k} + o\left(m^{2k}n^{2k}\right)
\end{flalign}
for even $k$ and
\begin{equation}
\label{eq:even_k}
\begin{array}{ll}
t_k(m,n) &\leq \displaystyle{t_2(m,n) \choose \lfloor{\frac{k}{2}}\rfloor}\frac{t(m,n)}{k} = \frac{t_2(m,n)^{\lfloor\frac{k}{2}\rfloor}t(m,n)}{\lfloor{\frac{k}{2}}\rfloor!k} + o\left(m^{2k}n^{2k}\right) \\
& = \displaystyle\frac{5^{k-1}6}{12^{\lfloor\frac{k}{2}\rfloor}\pi^{2k}\lfloor\frac{k}{2}\rfloor!k}m^{2k}n^{2k} + o\left(m^{2k}n^{2k}\right)
\end{array}
\end{equation}
for odd $k$.
Since for even $k$
$$
\frac{5^k}{12^{\frac{k}{2}}\pi^{2k}\frac{k}{2}!} < \frac{6^k}{\pi^{2k}k!}
$$
if and only if $k \leq 22$, and for odd $k$
$$
\frac{5^{k-1}6}{12^{\lfloor\frac{k}{2}\rfloor}\pi^{2k}\lfloor\frac{k}{2}\rfloor!k} < \frac{6^k}{\pi^{2k}k!}
$$
if and only if $k \leq 23$, we conclude that the upper bounds in (\ref{eq:odd_k}) and (\ref{eq:even_k}) improve estimation (\ref{eq:k_trivial_upper_bound}) for every $3 \leq k \leq 23$.



\bibliographystyle{siamplain}

\newpage
\appendix
\section{Estimations from \cref{th:s_general_formula}}
\medskip
\noindent
\subsection{Estimation of $\sum(m-\m)\sum(n-\n) \cdot |Z_2^b(\m,\n)|$}

Using \cref{lem:Z2b} and formulas (\ref{eq:sum_euler_gen}), (\ref{eq:sum_w}), and (\ref{eq:sum_phi_n_k}), we obtain
\begin{flalign}
\nonumber&\sum_{\m=1}^{m-1}(m-\m)\sum_{\substack{\n=1}}^{n-1}(n-\n) \cdot |Z_2^b(\m,\n)|\\
\nonumber&= \sum_{\m=1}^{m-1}(m-\m)\sum_{\substack{\n=1 \\ \n \perp \m}}^{n-1}(n-\n) \cdot |Z_2^b(\m,\n)|\\
\label{eq:symmetricity_f}
&= \sum_{\m=1}^{m-1}(m-\m)\sum_{\substack{\n=1 \\ \n \perp \m}}^{n-1}(n-\n)\left(\frac{1}{\pi^2}\m^2\n^2 + O(\m^2 \n + \m \n^2)\right)\\
\nonumber&= \frac{1}{\pi^2}\sum_{\m=1}^{m}(m\m^2-\m^3)\sum_{\substack{\n=1 \\ \n \perp \m}}^{n}(n\n^2-\n^3) + O(m^4 n^3+ m^3 n^4)\\
\nonumber&=\frac{1}{\pi^2}\sum_{\m=1}^{m}(m\m^2-\m^3)\left(\frac{\phi(\m)}{3\m}n^4-\frac{\phi(\m)}{4\m}n^4 + O\left(n^3 2^{w(\m)}\right)\right) + O(m^4 n^3+ m^3 n^4)\\
\nonumber&=\frac{1}{12\pi^2}\sum_{\m=1}^{m}\left(m n^4\m\phi(\m)- n^4\m^2\phi(\m) + O\left(m n^3\m^2 2^{w(\m)}\right)\right) + O(m^4 n^3+ m^3 n^4)\\
\nonumber&=\frac{1}{12\pi^2}\left(\frac{2m^4n^4}{\pi^2}- \frac{3m^4n^4}{2\pi^2} + O\left(m^4n^3\log m\right)\right)  + O(m^4 n^3+ m^3 n^4)\\
\nonumber&= \frac{1}{24\pi^4}m^4n^4 + O(m^4n^3\log m + m^3 n^4).
\end{flalign}
Now, symmetry of formula (\ref{eq:symmetricity_f}) implies also the estimation with a symmetric error term
\begin{flalign}
\nonumber\sum_{\m=1}^{m-1}(m-\m)\sum_{\substack{\n=1}}^{n-1}(n-\n) \cdot |Z_2^b(\m,\n)| = 
\frac{1}{24\pi^4}m^4n^4 + O\left(m^3n^4\log n + m^4n^3\right).
\end{flalign}
Finally, comparing the two estimations one can derive
\begin{flalign}
\label{eq:final_1}
\sum_{\m=1}^{m-1}(m-\m)\sum_{\substack{\n=1}}^{n-1}(n-\n) \cdot |Z_2^b(\m,\n)| = \frac{1}{24\pi^4}m^4n^4 + o(m^4n^4).
\end{flalign}

\noindent
\begin{minipage}{\textwidth}
\noindent
\subsection{Estimation of $\sum(m-\m)\sum(n-\n) \Big( |Z_2^c(\m,\n)| + |Z_1^a(\m,\n)| + |Z_0(\m,\n)| \Big)$}

Using Lemmas \ref{lem:Z2c}, \ref{lem:Z1a}, \ref{lem:Z0}, and formula (\ref{eq:sum_powers}), we obtain
\begin{flalign}
\nonumber&\sum_{\m=1}^{m-1}(m-\m)\sum_{\n=1}^{n-1}(n-\n)\Big( |Z_2^c(\m,\n)| + |Z_1^a(\m,\n)| + |Z_0(\m,\n)| \Big)\\
\nonumber&=\sum_{\m=1}^{m-1}(m-\m)\sum_{\n=1}^{n-1}(n-\n)\left(\frac{42}{\pi^4}\m^2\n^2 + \frac{72}{\pi^4}\m^2\n^2 + \frac{72}{\pi^4}\m^2\n^2 + o(\m^2\n^2)\right)\\
\nonumber&=\frac{186}{\pi^4}\sum_{\m=1}^{m}(m-\m)\m^2\sum_{\n=1}^{n}(n-\n)\n^2+ o(m^4n^4)\\
\nonumber&=\frac{186}{\pi^4}\sum_{\m=1}^{m}(m-\m)\m^2\left(\frac{n^4}{3} - \frac{n^4}{4} + O(n^3)\right)+ o(m^4n^4)\\
\label{eq:final_Z2_1_0}
&=\frac{31}{2\pi^4}n^4\sum_{\m=1}^{m}(m-\m)\m^2+ o(m^4n^4)=\frac{31}{24\pi^4}m^4n^4 + o(m^4n^4).
\end{flalign}
\end{minipage}

\medskip
\noindent
\begin{minipage}{\textwidth}
\noindent
\subsection{Estimation of $\sum(m-\m)\sum(n-\n) \cdot |Z_1^b(\m,\n)|$}

Using \cref{lem:Z1b} we derive
\begin{flalign}
\nonumber&\sum_{\m=1}^{m-1}(m-\m)\sum_{\n=1}^{n-1}(n-\n) \cdot |Z_1^b(\m,\n)|\\
\nonumber&=\sum_{\m=1}^{m-1}(m-\m)\sum_{\n=1}^{n-1}(n-\n)\left(\frac{6\n^2}{\pi^2}\sum_{\substack{b_1=1 \\ b_1 \perp \n}}^{\m}\left(2\m - b_1\right) + \frac{6\m^2}{\pi^2}\sum_{\substack{c_2=1 \\ c_2 \perp \m}}^{\n}\left(2\n - c_2\right) + O(\m^2 \n + \m \n^2)\right)\\
\nonumber&=\frac{6}{\pi^2}\sum_{\m=1}^{m}(m-\m)\sum_{\n=1}^{n}(n-\n)\n^2\sum_{\substack{b_1=1 \\ b_1 \perp \n}}^{\m}\left(2\m - b_1\right) \\
\label{eq:two_bit_sum}
&+ \frac{6}{\pi^2}\sum_{\n=1}^{n}(n-\n)\sum_{\m=1}^{m}(m-\m)\m^2\sum_{\substack{c_2=1 \\ c_2 \perp \m}}^{\n}\left(2\n - c_2\right) + o(m^4n^4).
\end{flalign}
\end{minipage}

We notice that the first of the summands in the latter formula is obtained from the second one by swapping 
$\m$ with $\n$, $b_1$ with $c_2$, and $m$ with $n$, hence it suffices to find a closed-form estimation only for one of them,
say for the first one.

Using formula (\ref{eq:sum_euler_gen}) we obtain
\begin{flalign}
\nonumber&\sum_{\n=1}^{n}(n-\n)\n^2\sum_{\substack{b_1=1 \\ b_1 \perp \n}}^{\m}\left(2\m - b_1\right)\\
\nonumber&=\sum_{\n=1}^{n}(n-\n)\n^2\left(2\frac{\phi(\n)}{\n}\m^2 - \frac{\phi(\n)}{2\n}\m^2 + O\left(\m2^{w(\n)}\right)\right)\\
\nonumber&=\frac{3}{2}\sum_{\n=1}^{n}\left(\m^2n\n\phi(\n) -\m^2\n^2\phi(\n) + O\left(\m n\n^22^{w(\n)}\right)\right)\\
\nonumber&=\frac{3}{2}\left(\m^2n\frac{2}{\pi^2}n^3 - \m^2\frac{3}{2\pi^2}n^4\right) + O(\m^2n^3 \log n) + O\left(\m n^4 \log n\right)\\
\label{eq:Z_1_b_error_1}
&=\frac{3}{4\pi^2}\m^2n^4 +n^2\left( O(\m^2n \log n) + O\left(\m n^2 \log n\right)\right).
\end{flalign}

By changing the order of summation in the above sum, we deduce the same result with a different error term:

\begin{flalign}
\nonumber&\sum_{\n=1}^{n}(n-\n)\n^2\sum_{\substack{b_1=1 \\ 
b_1 \perp \n}}^{\m}\left(2\m - b_1\right) =\sum_{b_1=1}^{\m}\left(2\m - b_1\right)\sum_{\substack{\n=1 \\ \n \perp b_1}}^{n}(n-\n)\n^2\\
\nonumber&=\sum_{b_1=1}^{\m}\left(2\m - b_1\right)\left(\frac{\phi(b_1)}{3b_1}n^4 - \frac{\phi(b_1)}{4b_1}n^4 + O\left(n^32^{w(b_1)}\right)\right)\\
\nonumber&=\frac{1}{12}\sum_{b_1=1}^{\m}\left(2\m n^4\frac{\phi(b_1)}{b_1} - n^4\phi(b_1) + O\left( n^3\m2^{w(b_1)}\right)\right)\\
\nonumber&=\frac{1}{12}\left(2\m n^4\frac{6}{\pi^2}\m + O(\m n^4 \log \m) - n^4\frac{3}{\pi^2}\m^2 + O(\m n^4 \log \m) + O(n^3\m^2 \log \m )\right)\\
\label{eq:Z_1_b_error_2}
&=\frac{3}{4\pi^2}\m^2 n^4 + n^2\left(O(\m n^2 \log \m) + O\left(n\m^2 \log \m \right)\right).
\end{flalign}

Comparing the error terms in (\ref{eq:Z_1_b_error_1}) and (\ref{eq:Z_1_b_error_2}) we obtain
\begin{flalign}
\nonumber\sum_{\n=1}^{n}(n-\n)\n^2\sum_{\substack{b_1=1 \\ b_1 \perp \n}}^{\m}\left(2\m - b_1\right) = \frac{3}{4\pi^2}\m^2 n^4 + o(\m^2n^4).
\end{flalign}

Using the obtained formula and formula (\ref{eq:sum_powers}) we proceed
\begin{flalign}
\nonumber\frac{6}{\pi^2}\sum_{\m=1}^{m}(m-\m)\sum_{\n=1}^{n}(n-\n)\n^2\sum_{\substack{b_1=1 \\ b_1 \perp \n}}^{\m}\left(2\m - b_1\right) &=\frac{6}{\pi^2}\sum_{\m=1}^{m}(m-\m)\left(\frac{3}{4\pi^2}\m^2 n^4 + o(\m^2n^4)\right)\\
\nonumber&=\frac{9n^4}{2\pi^4}\sum_{\m=1}^{m}(m \m^2-\m^3)+ o(m^4n^4)\\
\nonumber&=\frac{9n^4}{2\pi^4}\left(\frac{m^4}{3} - \frac{m^4}{4}
\right)+ o(m^4n^4)\\
\nonumber&=\frac{3}{8\pi^4}m^4n^4+ o(m^4n^4).
\end{flalign}

Due to symmetry, the second summand in formula (\ref{eq:two_bit_sum}) has the same asymptotics, and therefore

\begin{flalign*}
\sum_{\m=1}^{m-1}(m-\m)\sum_{\n=1}^{n-1}(n-\n) \cdot |Z_1^b(\m,\n)| = \frac{3}{4\pi^4}m^4n^4+ o(m^4n^4).
\end{flalign*}

\end{document}